\documentclass{article}
\usepackage{amsmath}
\usepackage{amssymb}
\usepackage{titlesec}
\usepackage{graphicx}
\usepackage{fancyhdr}
\usepackage{latexsym}
\usepackage{mathrsfs}
\usepackage{booktabs}
\usepackage{mathrsfs}
\usepackage{mathtools}
\usepackage{amsthm}
\usepackage{wasysym}
\usepackage{cite}
\usepackage{color}
\usepackage{amsfonts}
\usepackage{amsmath,amssymb}
\usepackage{graphics}
\usepackage{multimedia}
\usepackage{graphics}
\usepackage{cite}
\usepackage{latexsym}
\usepackage{amsmath}
\usepackage{amssymb}
\usepackage[dvips]{epsfig}
\usepackage{amscd}
\numberwithin{equation}{section}

\newtheorem{theorem}{Theorem}[section]
\newtheorem{lemma}{Lemma}[section]
\newtheorem{remark}{Remark}[section]

\title{The Lifespan of Strong Solutions to the Compressible MHD Equations with Entropy Transport in the Presence of Vacuum}
\author{Yongteng G{\small U}$^{a}$\thanks{Email addresses: guyongteng24@mails.ucas.ac.cn (Y. T. Gu), xdhuang@amss.ac.cn (X. D. Huang). }, Xiangdi H{\small UANG}$^{a}$  \\ 
	{\normalsize a. State Key Laboratory of Mathematical Sciences, Academy of Mathematics and Systems Sciences,}\\
	{\normalsize Chinese Academy of Sciences, Beijing 100190, China;}\\
}
\date{}

\begin{document}
	
	\maketitle
	
	\begin{abstract}
		In this paper, we investigate the finite time blow-up of strong solutions to the compressible magnetohydrodynamic (MHD) system (without magnetic diffusion) coupled with entropy transport, and derive an upper bound for the lifespan of such solutions. We first establish the local well-posedness of strong solutions for bounded domains and study the mechanism of finite-time singularity formation in the 2D radially symmetric case and 3D cylindrically symmetric case. We prove that if the initial density vanishes in an interior region containing the origin and the magnetic field is non-trivial within this vacuum region, the strong solution must blow up in finite time. These results generalize and improve the previous results of Huang-Xin-Yan [Math. Ann. 392 (2025) 2365–2394] for the compressible isentropic MHD equations. Significantly, we extend this blow-up result to the free boundary problem. Our analysis of the boundary's expansion allows us to explicitly estimate the maximum lifespan of the solution.
	
		Keywords: Compressible MHD equations; Entropy transport; Strong solutions; Vacuum; Finite time blow-up. 
	\end{abstract}
	\section{Introduction}
	In this paper, we investigate the compressible magnetohydrodynamic (MHD) system (without magnetic diffusion) coupled with entropy transport. This system can be derived as a specific limit of the general Navier-Stokes-Fourier system for atmospheric flows discussed by Klein \cite{klein}. As a simplification of the Navier-Stokes-Fourier system, we assume that thermal conduction is absent and viscous dissipation is negligible. By further neglecting external geophysical forces such as gravity, we arrive at the following system
\begin{equation}\label{1.1}
	\begin{cases}
		\rho _t+\nabla \cdot \left( \rho u \right) =0,\\
		\left( \rho u \right) _t+\nabla \cdot \left( \rho u\otimes u \right) -\mu \Delta u-\left( \mu +\lambda \right) \nabla \left( \nabla \cdot u \right) +\nabla P\left( \rho ,\tilde{s} \right) =\left( \nabla \times B \right) \times B,\\
		\partial _t\left( \rho \tilde{s} \right) +\nabla \cdot \left( \rho \tilde{s} u \right) =0,\\
		B_t-\nabla \times \left( u\times B \right) =0,\\
		\nabla \cdot B=0,\\
	\end{cases}
\end{equation}
Here, the unknown functions $\rho,u, P, \tilde{s}, B$ denote the ﬂuid density, velocity field, pressure, entropy and magnetic field, respectively. The viscosity coefficients $\mu$ and $\lambda$ satisfy the following physical restrictions
$$\mu >0,\quad \frac{2}{d}\mu +\lambda \ge 0,$$
where $d$ is the spatial dimension. The pressure $P$ is given by the state equation
$$P\left( \rho ,\tilde{s} \right) =\rho ^{\gamma}\varGamma \left( \tilde{s} \right), \quad \gamma>1,$$ 
where $\Gamma(\tilde{s})$ denotes a positive, smooth, and strictly monotone function defined on $\mathbb{R}_+$.
\par
By neglecting the effect of entropy and assuming the pressure law $P = A\rho^\gamma$, we find that equation \eqref{1.1} reduces to the compressible MHD equations. This system has attracted extensive research attention concerning its well-posedness theory. The local existence of strong solutions to the compressible MHD equations with large initial data was established by Vol'pert and Hudjaev \cite{Vol'pert-Hudjaev}, under the assumption that the initial density remains strictly positive. Recently, Fan-Gu-Huang \cite{Fan-Gu-Huang} proved the existence of global large solutions to the three-dimensional compressible magnetohydrodynamic (MHD) equations under the assumption that the viscosity coefficients satisfy  $\mu \left( \rho \right) =\mu \rho ^{\alpha},\lambda \left( \rho \right) =\lambda \rho ^{\alpha}$, provided that the initial density is sufficiently large. In the presence of a vacuum, the local existence of unique strong solutions for the full compressible MHD system has been established in \cite{Fan-Yu}, provided that the initial data satisfy a compatibility condition. Based on work by Huang-Li-Xin \cite{Huang-Li-Xin} regarding Navier-Stokes equations, Li-Xu-Zhang \cite{Li-Xu-Zhang} established the global well-posedness of classical solutions for the 3D isentropic compressible MHD equations, where the initial data requires small energy but allows for large oscillations and vacuum. In addition, a lot of progress has been made regarding models with entropy transport. For instance, Michálek \cite{Michálek} proved the stability of the compressible Navier–Stokes system with an entropy transport equation. Maltese et al. \cite{Maltese-Michálek-Novotný } proved the existence of globally defined weak solutions on condition that the adiabatic
constant $\gamma>3/2$, which corresponds to the classical result of Feireisl et al. \cite{Feireisl-Novotný-Petzeltová} for the isentropic case. Later, the authors in \cite{M. Lukáová-Medvid’ová-Andreas Schömer} established a blow-up criterion for strong solutions in terms of the $L^\infty$-norms of the density and velocity.
\par
 Having reviewed the results relevant to our model, our main focus in this paper is to study whether local solutions to the system with entropy transport blow up in finite time for general large initial data. An early study on the finite time blow-up of compressible Navier-Stokes equations was conducted by Xin \cite{Xin}, who investigated solutions in the space $C^1([0,\infty), H^m(\mathbb{R}^d))$ (with $m>[d/2]+2$) and found that they blow up in finite time when the initial density has compact support. Later, improving upon Xin’s previous result, Xin-Yan\cite{Xin-Yan} established that any classical solution for viscous compressible fluids without heat conduction inevitably blows up in finite time, provided that the initial data have an isolated mass group. Nevertheless, the above results were established assuming that local solutions are sufficiently smooth. Actually, finding such local solutions is difficult, or they might not exist at all. Li-Wang-Xin \cite{Li-Wang-Xin} proved that classical solutions with finite energy do not exist in the inhomogeneous Sobolev space for any short time under certain natural assumptions on the initial data near a vacuum. Therefore, there has been growing research interest in finding solutions that are known to exist locally but blow up in finite time. Recently, it was proved by Merle et al. \cite{Merle-Raphael-p-J} that under conditions of strictly positive density, there exist smooth solutions to the 3D spherically symmetric compressible Navier-Stokes equations that inevitably form a shell singularity in finite time. For cases with interior vacuum, Huang-Xin-Yan \cite{Huang-Yan-Xin} established a class of local strong solutions that exist locally and blow up in finite time by analyzing the balance between viscous stress and the magnetic Lorentz force.
\par
Based on the preceding discussion, we investigate the finite time blow-up of strong solutions to the system \eqref{1.1} in the presence of an interior vacuum. Furthermore, we consider the lifespan of solutions to the free boundary problem for the system \eqref{1.1} under the assumption that local strong solutions exist.
\par
To handle the strong nonlinearity of the pressure, we combine the renormalized mass and entropy equations finding that $P$ satisfies
$$	P_t+u\cdot \nabla P+\gamma P\nabla \cdot u=0.$$
Consequently, system \eqref{1.1} can be reformulated in terms of the variables $(\rho, u, P, B)$ as follows
\begin{equation}\label{1.2}
	\begin{cases}
		\rho _t+\nabla \cdot \left( \rho u \right) =0,\\
		\left( \rho u \right) _t+\nabla \cdot \left( \rho u\otimes u \right) -\mu \Delta u-\left( \mu +\lambda \right) \nabla \left( \nabla \cdot u \right) +\nabla P=\left( \nabla \times B \right) \times B,\\
		P_t+u\cdot \nabla P+\gamma P\nabla \cdot u=0,\\
		B_t-\nabla \times \left( u\times B \right) =0,\\
		\nabla \cdot B=0.\\
	\end{cases}
\end{equation}
Based on this reformulated system (1.2), we will state the main results of this paper in the following section.
	\section{Main results}
	\subsection{Local well-posedness}
We begin by establishing the local well-posedness of System \eqref{1.2} in general bounded domains. Theorem \ref{th1} serves as a foundation for the analysis of finite time blow-up under radial symmetry presented later.
	\begin{theorem}\label{th1}(Local existence of strong solutions-2D case)
		Let $\Omega \subset \mathbb{R} ^2$ be a $2D$ bounded domain with smooth boundary.	Assume that the initial data $(\rho_0, u_0, P_0, B_0)$ satisfy the following conditions
		\begin{equation}\label{1.3}
			\begin{cases}
				0 \le \rho_0, P_0 \in W^{1,q}, \quad B_0 \in W^{1,q}, \quad \nabla \cdot B_0 = 0 \text{ in } \Omega, \\
				u_0 \in H^2(\Omega), \quad u_0|_{\partial \Omega} = 0,
			\end{cases}
		\end{equation}
		for some constant $q>2$, and satisfy the compatibility condition:
		\begin{equation}\label{eq:compatibility}
			-\mu \Delta u_0 - (\mu + \lambda) \nabla (\nabla \cdot u_0) + \nabla P_0 + B_0 \times (\nabla \times B_0) = \rho_0^{1/2} g, \quad \text{in } \Omega,
		\end{equation}
		for some $g \in L^2(\Omega)$. Then, there exists a small time $T^* > 0$ and a unique strong solution $(\rho, u, P, B)$ to the initial boundary value problem \eqref{1.2} and \eqref{1.3} such that
		\begin{equation}\label{11.15}
			\begin{cases}
				(\rho, P, B) \in C([0, T^*]; W^{1,q}), \quad (\rho_t, P_t, B_t) \in C([0, T^*]; L^q), \\
				u \in C([0, T^*]; H^2) \cap L^2(0, T^*; W^{2,q}), \\
				u_t \in L^2(0, T^*; H^1), \quad \sqrt{\rho} u_t \in L^\infty(0, T^*; L^2).
			\end{cases}
		\end{equation}
	\end{theorem}
\begin{remark}
The local existence result mentioned above also holds for the three-dimensional case, provided that the range of $q$ is adjusted to $3 < q \le 6$. The proof is entirely analogous to that of the 2D case. This ensures the existence of the local solutions required to demonstrate finite time blow-up in the axisymmetric cylinders.
\end{remark}
\subsection{Finite time blow-up for 2D Disk}
Now we consider the radially symmetric solutions of \eqref{1.2}. In this case, we investigate solutions within the class of radially symmetric functions $(\rho, \vec{u}, P, \vec{B})$ defined by
\begin{equation*}
		\rho(x,t) = \rho(r,t), \quad \vec{u}(x,t) = \frac{(x_1,x_2)}{r} u(r,t),\quad P(x,t)=P(r,t), \quad \vec{B}(x,t) = \frac{(-x_2, x_1)}{r} B(r,t).
\end{equation*}
Here $(\rho,u,P,B)$ are scalar functions. \par
In radially symmetric coordinates, the original system \eqref{1.2} can be written as
	\begin{equation}\label{11.16}
	\begin{cases}
		\rho_t + (\rho u)_r + \frac{\rho u}{r} = 0, \\[8pt]
		\rho (u_t + u u_r) + P_r = (2\mu + \lambda) \left( u_r + \frac{u}{r} \right)_r - B \left( B_r + \frac{B}{r} \right), \\[8pt]
		P_t + u P_r + \gamma P \left( u_r + \frac{u}{r} \right) = 0, \\[8pt]
		B_t + (u B)_r = 0,
	\end{cases}
\end{equation}
with $0 \le r \le R_0< \infty$ and the initial data
\begin{equation}
	\rho \left( r,0 \right) =\rho _0\left( r \right) ,\quad  u\left( r,0 \right) =u_0\left( r \right) ,\quad P\left( r,0 \right) =P_0\left( r \right) ,\quad B\left( r,0 \right) =B_0(r),
\end{equation}
and the Dirichlet boundary condition 
\begin{equation}
	u\left( R_0,t \right) =0 ,\,t\ge 0.
\end{equation}
Note also that the continuity of the velocity and magnetic field in the center will force
\begin{equation}\label{11.10}
	u(0, t) = B(0, t) = 0, \quad t \ge 0.
\end{equation}
Given the existence of local strong solutions in general bounded domains, we turn to the blow-up problem within a radially symmetric framework. This work presents the first proof of finite-time singularity formation for the simplified full Navier-Stokes equations in a regime where local strong solutions exist. 
\begin{theorem}\label{Thm 1.2}
	(Blow-up of Strong Solutions for 2D Disk)
	Let $\Omega = B_{R_0}$ denote a two-dimensional disk of radius $R_0$ centered at the origin. Suppose the initial data $(\rho_0, \vec{u}_0, P_0,\vec{B}_0)$ for the initial-boundary value problem \eqref{11.16}-\eqref{11.10} satisfy the regularity conditions
	\begin{equation} \label{eq:3.6}
		\begin{cases}
			(\rho_0, B_0) \in W^{1,q}(\Omega), \quad \rho_0 \ge 0, \\
			\vec{u}_0 \in H^2(\Omega), \quad \vec{u}_0|_{\partial\Omega} = 0,
		\end{cases}
		\quad \text{and} \quad \nabla \cdot \vec{B}_0 = 0 \text{ in } \Omega,
	\end{equation}
	for some constant $q > 2$, along with the compatibility condition:
	\begin{equation} \label{eq:3.7}
		-\mu \Delta \vec{u}_0-(\mu +\lambda )\nabla (\nabla \cdot \vec{u}_0)+\nabla P_0+\vec{B}_0\times (\nabla \times \vec{B}_0)=\rho_0 ^{1/2}\vec{g},\quad \mathrm{in} \Omega ,
	\end{equation}
	for some $\vec{g} \in L^2$. Under these conditions, Theorem \ref{th1} guarantees the existence of a unique local strong solution $(\rho, \vec{u}, P, \vec{B})$ on a time interval $[0, T^*]$.
	
	Assume further that the initial data admit the following radial symmetry
	\begin{equation} \label{11.19}
		\rho_0(\vec{x}) = \rho_0(r), \quad \vec{u}_0(\vec{x}) = \frac{\vec{x}}{r} u_0(r), \quad P_0(\vec{x})=P_0(r),\quad \vec{B}_0(\vec{x}) = \frac{(-x_2, x_1)}{r} B_0(r).
	\end{equation}
	Then, the unique strong solution preserves this symmetry, taking the form
	\begin{equation} \label{eq:3.9}
		\rho(\vec{x},t) = \rho(r,t), \quad \vec{u}(\vec{x},t) = \frac{\vec{x}}{r} u(r,t),\quad P(\vec{x},t)=P(r,t), \quad \vec{B}(\vec{x},t) = \frac{(-x_2, x_1)}{r} B(r,t),
	\end{equation}
	where $\rho, u, P, B$ are scalar functions.
	
	If there exists a radius $r_0 \in (0, R_0)$ such that the initial density and pressure vanish in the inner ball $B_{r_0}$, i.e.
	\begin{equation} \label{eq:3.10}
		\rho_0(r) = 0,\,P_0(r)=0 \quad \text{for } r \in [0, r_0], 
	\end{equation}
	while the magnetic field satisfies the non-degeneracy condition:
	\begin{equation} \label{eq:3.11}
		\int_0^{r_0} B_0(r) \, dr \ne 0, 
	\end{equation}
	then the local strong solution cannot exist globally and must blow up in finite time.
	
	Specifically, for any constant $\alpha \in (1, 2)$, the lifespan $T$ of the solution is bounded by
	\begin{equation} \label{eq:3.12}
	T\le \left( \frac{1}{\sqrt{2\mu +\lambda}R_0}\frac{(2-\alpha )^2\left| \int_0^{r_0}{B_0}(r)\,dr \right|^2}{2(\frac{\alpha}{\sqrt{2\alpha -2}}+\frac{\alpha +1}{\sqrt{2\alpha}})} \right) ^{-2}E_0,
	\end{equation}
	where $E_0$ denotes the initial total energy
	\begin{equation}
	E_0=\frac{1}{2}\int_{\Omega}{\left( \rho _0|\vec{u}_0|^2+|\vec{B}_0|^2 \right)}dx+\int_{\Omega}{\frac{P}{\gamma -1}}dx.
	\end{equation}
\end{theorem}
\begin{remark}
	Compared to \cite{Huang-Yan-Xin}, the model we consider is more general, allowing density and pressure to be independent variables. However, to analyze the balance between viscous stress and the magnetic Lorentz force, we require that the initial pressure also vanishes within the interior vacuum region i.e. $	\rho_0(r) = 0,\,P_0(r)=0 $, $\text{for } r \in [0, r_0]$. Moreover, considering the state equation $P=\rho ^{\gamma}\varGamma \left( \tilde{s} \right)$, this definition is entirely natural and reasonable.
\end{remark}
\begin{remark}
	According to \cite{Xin-Yan}, classical solutions inevitably blow up if the initial density contains an isolated mass group, implying the presence of an annular vacuum region. In contrast, our method of proof necessitates that the initial density contains an interior vacuum region; an annular vacuum region does not work for our approach.
\end{remark}
\subsection{Finite time Blow-up for 3D Axisymmetric Cylinders}
We now turn our attention to the solutions of \eqref{1.2} under the assumption of axisymmetry
\begin{equation*}
	\begin{split}
			&\rho (\vec{x},t)=\rho (r,t), \quad P(x,t)=P(r,t),
		\\
		&\vec{u}(x,t)=u(r,t)\frac{(x_1,x_2,0)}{r}+v\left( r,t \right) \frac{\left( -x_2,x_1,0 \right)}{r}+w\left( r,t \right) \left( 0,0,1 \right) ,
		\\
		&\vec{B}(x,t)=\frac{(-x_2,x_1,0)}{r}B(r,t).
	\end{split}
\end{equation*}
Here $(\rho,u,v,w,P,B)$ are scalar functions. Now we consider a domain $\Omega=\{\vec{x}\,| \,x_{1}^{2}+x_{2}^{2}< R_{0}^{2}, x_3 \in \mathbb{T}^1\}$, which represents a cylinder of radius $R_0$ with periodic boundary conditions in the axial direction $x_3$ with period 1.\par
In cylindrically symmetric coordinates, the original system \eqref{1.2} can be written as
\begin{equation}\label{12.16}
	\begin{cases}
			\rho_t + (\rho u)_r + \frac{\rho u}{r} = 0, \\
			\rho \left( u_t + u u_r - \frac{v^2}{r} \right) + P_r = (2\mu + \lambda) \left( u_r + \frac{u}{r} \right)_r - B \left( B_r + \frac{B}{r} \right), \\
			\rho \left( v_t + u v_r + \frac{uv}{r} \right) = \mu \left( v_r + \frac{v}{r} \right)_r, \\
			\rho \left( w_t + u w_r \right) = \mu \frac{1}{r}(r w_r)_r, \\
			P_t + u P_r + \gamma P \left( u_r + \frac{u}{r} \right) = 0, \\
			B_t + (u B)_r = 0.
	\end{cases}
\end{equation}
The initial data are given by
\begin{equation}
	\begin{split}
		\rho \left( r,0 \right) =\rho _0\left( r \right) ,\quad  &u\left( r,0 \right) =u_0\left( r \right) ,\quad v\left( r,0 \right) =v_0\left( r \right), \quad w\left( r,0 \right) =w_0\left( r \right),
		\\
		&P\left( r,0 \right) =P_0\left( r \right) ,\quad B\left( r,0 \right) =B_0(r),
	\end{split}
\end{equation}
and the Dirichlet boundary condition 
\begin{equation}
	u\left( R_0,t \right) =v\left( R_0,t \right)=w\left( R_0,t \right)=0 ,\,t\ge 0.
\end{equation}
Note also that the continuity of the velocity and magnetic field in the center will force
\begin{equation}\label{12.19}
	u(0, t) = v(0,t)=B(0, t) = 0, \quad t \ge 0.
\end{equation}
\begin{theorem}\label{th1.3}(Blow-up of Strong Solutions for 3D Axisymmetric Cylinders)	
	Let $\Omega$ be a cylinder with radius $R_0$ which is periodic in the $x_3$-direction with period 1. Suppose the initial data $(\rho_0, \vec{u}_0, P_0,\vec{B}_0)$ for the initial-boundary value problem \eqref{12.16}-\eqref{12.19} satisfy the regularity conditions
	\begin{equation} \label{eq:33.6}
		\begin{cases}
			(\rho_0, B_0) \in W^{1,q}(\Omega), \quad \rho_0 \ge 0, \\
			\vec{u}_0 \in H^2(\Omega), \quad \vec{u}_0|_{\partial\Omega} = 0,
		\end{cases}
		\quad \text{and} \quad \nabla \cdot \vec{B}_0 = 0 \text{ in } \Omega,
	\end{equation}
	for some constant $3<q \le 6$, along with the compatibility condition
	\begin{equation}
		-\mu \Delta \vec{u}_0-(\mu +\lambda )\nabla (\nabla \cdot \vec{u}_0)+\nabla {P}_0+\vec{B}_0\times (\nabla \times \vec{B}_0)=\rho_0 ^{1/2}\vec{g},\quad \mathrm{in} \Omega ,
	\end{equation}
	for some $\vec{g} \in L^2$. Under these conditions, Theorem \ref{th1} can be extended to guarantee the existence of a unique local strong solution $(\rho, \vec{u}, P, \vec{B})$ on the time interval $[0, T^*]$ for the three-dimensional domain $\Omega=\{\vec{x}\,| \,x_{1}^{2}+x_{2}^{2}< R_{0}^{2}, x_3 \in \mathbb{T}^1\}$.
	
	Assume further that the initial data admit the following cylindrical symmetry
	\begin{equation} 
		\begin{split}
				\rho _0(\vec{x})=\rho _0(r),\quad &\vec{u}_0(\vec{x})=\frac{\left( x_1,x_2,0 \right)}{r}u_0(r)+\frac{\left( -x_2,x_1,0 \right)}{r}v_0\left( r \right) +\left( 0,0,1 \right) w_0\left( r \right) ,
			\\
		 &P_0(\vec{x})=P_0(r),\quad \vec{B}_0(\vec{x})=\frac{(-x_2,x_1,0)}{r}B_0(r).
		\end{split}
	\end{equation}
	Then, the unique strong solution preserves this symmetry, taking the form
	\begin{equation} 
		\begin{split}
			\rho (\vec{x},t)=\rho (r,t),\quad &\vec{u}(\vec{x},t)=\frac{\left( x_1,x_2,0 \right)}{r}u(r,t)+\frac{\left( -x_2,x_1,0 \right)}{r}v\left( r,t \right) +\left( 0,0,1 \right) w\left( r,t \right) ,
			\\
			&P(\vec{x},t)=P(r,t),\quad \vec{B}(\vec{x},t)=\frac{(-x_2,x_1,0)}{r}B(r,t),
		\end{split}
	\end{equation}
	where $\rho, u, P, B$ are scalar functions.
	
	If there exists a radius $r_0 \in (0, R_0)$ such that the initial density and pressure vanish in the inner cylinder $\Omega_1=\{\vec{x}\,| \,x_{1}^{2}+x_{2}^{2}< r_{0}^{2}, x_3 \in \mathbb{T}^1\}$, i.e.
	\begin{equation} \label{eq:33.10}
		\rho_0(r) = 0,\, P_0(r)=0 \quad \text{for } r \in [0, r_0], 
	\end{equation}
	while the magnetic field satisfies the non-degeneracy condition
	\begin{equation} \label{eq:33.11}
		\int_0^{r_0} B_0(r) \, dr \ne 0, 
	\end{equation}
	then the local strong solution cannot exist globally and must blow up in finite time.
	
	Specifically, for any constant $\alpha \in [\frac{7}{6}, 2)$, the lifespan $T$ of the solution is bounded by
	\begin{equation} \label{eq:33.12}
T\le \left( \frac{1}{\sqrt{2\mu +\lambda}R_0}\frac{(2-\alpha )^2\left| \int_0^{r_0}{B_0}(r)\,dr \right|^2}{2\sqrt{2}(\frac{\alpha}{\sqrt{2\alpha -2}}+\frac{\alpha +1}{\sqrt{2\alpha}})} \right) ^{-2}E_0,
	\end{equation}
	where $E_0$ denotes the initial energy
	\begin{equation}
	E_0=\int_0^{R_0}{\left( \frac{\rho_0}{2}\left( u_0^2+v_0^2+w_0^2 \right) r+\frac{P_0}{\gamma -1}r+\frac{1}{2}B_0^2r \right) dr}.
	\end{equation}
\end{theorem}
\begin{remark}
	We provide examples of blow-up in the three-dimensional case. In fact, the axisymmetric case corresponds exactly to that of a two-dimensional radially symmetric velocity field with swirl. Similarly, it can be proven that the equations governing a two-dimensional velocity field with swirl will also blow up in finite time.
\end{remark}
\begin{remark}
Since the initial data for the system \eqref{12.16} depend only on $r$, the solution is independent of $x_3$ and thus automatically satisfies the periodic boundary condition in the $x_3$-direction. Consequently, the existence of a local strong solution for the cylinder with radius $R_0$ (periodic in $x_3$) follows as a direct corollary of the existence theory for three-dimensional domains.
\end{remark}
\begin{remark}
We require $\alpha \in [\frac{7}{6}, 2)$ since, in the three-dimensional case, we only have the embedding $H^2 \hookrightarrow W^{1,6}$.
\end{remark}
\subsection{Finite time Blow-up for free boundary problem}
With the local existence and blow-up results for fixed boundary established, we now turn our attention to the dynamic behavior of the system with a free boundary. Specifically, we investigate the evolution of a radially symmetric strong solution governed by a free boundary condition. In this case, the fluid domain is time-dependent, denoted by $\Omega_t = \{x : 0 \le |x| \le a(t)\}$, where the boundary radius $a(t)$ is driven by the fluid velocity. The following theorem establishes that, even in this free boundary framework, finite-time singularity formation is inevitable under the presence of an interior vacuum and a non-trivial magnetic field.
\begin{theorem}\label{thm 1.4}
	(Blow-up of free boundary problem) Assume that $(\rho,\vec{u},P,\vec{B})$ is a 2D radially symmetric strong solution \textup{(}as defined in \eqref{11.15}\textup{)} of \eqref{11.16} with the initial data \eqref{11.19}  and the boundary condition
	\begin{equation}\label{11.25}
		u\left( 0,t \right) =B\left( 0,t \right) =0, \quad \left( \frac{B^2}{2}+P-\left( 2\mu +\lambda \right) \left( u_r+\frac{u}{r} \right) \right) \left( a\left( t \right) ,t \right) =0,
	\end{equation}
where $a(t)$ satisfies 
\begin{equation}\label{11.26}
	a'\left( t \right) =u\left( a\left( t \right) ,t \right) , \quad a(0)=a_0,
\end{equation}
corresponding to the free surface of $\Omega_t$, where 
$$\Omega _t\overset{\mathrm{def}}{=}\left\{ x\,|\,0\le \left| x \right|\le a\left( t \right) \right\}.$$
Furthermore, if there exists a constant $r_0 \in (0,a(0))$ such that 
\begin{equation}
	\rho_0(r) = 0,\,P_0(r)=0 \quad \text{for } r \in [0, r_0], 
\end{equation}
and 
	\begin{equation}
	\int_0^{r_0} B_0(r) \, dr \ne 0,
\end{equation}
then the strong solution $(\rho,\vec{u},P,\vec{B})$ will blow up in finite time. Moreover, for any constant $\alpha \in (1,2)$, the lifespan $T$ of the solution is bounded by
\begin{equation}
	T\le \exp \left( E_0\left( \frac{\left| \int_0^{r_0}{B_0}(r)\,dr \right|^2(2-\alpha )^2}{2\sqrt{2\mu +\lambda}C\left[ \frac{\alpha}{\sqrt{2\alpha -2}}+\frac{\alpha +1}{\sqrt{2\alpha}} \right]} \right) ^{-2} \right)-1,
\end{equation}
where $C$ is a constant depending on $a_0,E_0,\mu,\lambda$.
\end{theorem}
\begin{remark}
Physically, this boundary condition requires the effective viscous flux with the magnetic field, $F=P+\frac{1}{2}|B|^2-(2\mu + \lambda) \operatorname{div} u$, to vanish at the boundary. Under this condition, we are able to prove the fundamental energy inequality.
\end{remark}
\begin{remark}
	It is important to note that the local well-posedness of strong solutions for this free boundary problem remains an open question. The primary difficulty lies in the fact that existing literature typically considers fluids that are strictly bounded away from vacuum or physical vacuum models where the density vanishes only at the boundary. In contrast, our model involves the presence of an interior vacuum. Consequently, the blow-up result presented here is a conditional one.
\end{remark}
\begin{remark}
    When the pressure law is simplified to $P=A\rho^{\gamma}$ , corresponding to the model studied in \cite{Huang-Yan-Xin}, this conclusion remains valid. In this case, the pressure naturally vanishes in vacuum regions.
\end{remark}
\subsection{Sketch of the proof}
We will prove the aforementioned theorems in three sections. In Section 3, we prove the blow-up of strong solutions by using a fractional moment argument. In the vacuum region, we initially obtain not the balance between viscous stress and the magnetic Lorentz force, but rather
\begin{equation}
	(2\mu + \lambda) \partial_r \nabla \cdot \vec{u} = B \left( B_r + \frac{B}{r} \right)+P_r \quad \text{a.e.}
\end{equation} 
Since we define $P_0(r)=0$ when $\rho_0(r) = 0$ for $r \in [0, r_0]$. The particle path starting from $r_0$ satisfies the particle trajectory equation 
$$\begin{cases}
		\frac{\partial \vec{X}}{\partial t}(x,t) = \vec{u}(X(x,t), t), \\
		\vec{X}(r_0,0) =r_0.
\end{cases}$$
Let $R(t)$ be a particle path at time $t$ in the region of interior vacuum starting from point $r_0$. 
By the uniqueness of particle paths and the transport equation representation for density and pressure, we obtain
\begin{equation}
	\rho = 0, P=0 \text{ on } [0, R(t)], \quad R(0) = r_0.
\end{equation}
Therefore, in the vacuum region, we arrive at 
\begin{equation}
	(2\mu + \lambda) \partial_r \nabla \cdot \vec{u} = B \left( B_r + \frac{B}{r} \right) \quad \text{a.e.}
\end{equation}
Consequently, we can employ the test function method used in \cite{Huang-Yan-Xin} to obtain a lower bound of $\|\nabla \cdot \vec{u}\|_{L^2}$, thereby deriving finite time blow-up. In Section 4, we extend the result to the free boundary problem. First, we prove that the free boundary problem, under appropriate boundary conditions, satisfies energy conservation. Next, we fully utilize the temporal growth estimate of the free boundary under radially symmetric conditions. Then, based on the fact that the inner vacuum boundary cannot overtake the outer free boundary, we obtain a time-weighted estimate of $\|\nabla \cdot \vec{u}\|_{L^2}$. We find that the growth rate of the free boundary required for the solution to blow up is precisely critical
\begin{equation}
	\int_0^T{\frac{1}{1+t}}\,dt\lesssim \int_0^T{\frac{1}{a^2\left( t \right)}}\,dt\lesssim \int_0^T{\left\| \nabla \cdot \vec{u} \right\| _{L^2}^{2}}\,dt\lesssim E_0\,\,\Rightarrow \,\,T<\infty ,
\end{equation}
due to 
\begin{equation}
	a(t) \le C(1+t)^{1/2}.
\end{equation}
In Section 5, we establish the local well-posedness of System \eqref{1.2}. We prove the existence of local strong solutions by employing linearization of the equations and iterative techniques. This guarantees the existence of the local solutions considered for fixed boundary.
\section{Finite time blow-up for fixed boundary}\label{s4}
	\subsection{Dynamic behavior in the vacuum region}
	By Theorem \ref{th1}, let $(\rho, \vec{u},P,\vec{B})$ be the local radially symmetric strong solution to the problem \eqref{11.16}.
	 This solution satisfies the following regularities
	\begin{equation*}
		(\rho, P,\vec{B}) \in C([0,T]; W^{1,q}), \quad \vec{u} \in C([0,T]; H^2) \cap L^2(0,T; W^{2,q}).
	\end{equation*}
Given that $\nabla \vec{u} \in L^2(0,T; L^{\infty})$ due to the Sobolev embedding, we obtain a unique family of particle trajectories $\vec{X}(x,t)$ satisfying
\begin{equation} \label{eq:3.15}
	\begin{cases}
		\frac{\partial \vec{X}}{\partial t}(x,t) = \vec{u}(X(x,t), t), \\
		\vec{X}(x,0) = x.
	\end{cases}
\end{equation}
 Note that the initial data satisfy 
	$$\rho_0(r) = 0,\,P_0(r)=0 \quad \text{for } r \in [0, r_0]. $$
	
Let $R(t)$ be a particle path at time $t$ in the region of interior vacuum starting from point $r_0$. Due to the transport equation representation
\begin{equation}
	\begin{split}
		\rho \left( \vec{X}\left( x,t \right) ,t \right) =\rho _0\left( x \right) \exp \left\{ \int_0^t{\left( -\nabla \cdot \vec{u}\left( \vec{X}\left( x,s \right) ,s \right) ds \right)} \right\} ,
		\\
		P\left( \vec{X}\left( x,t \right) ,t \right) =P_0\left( x \right) \exp \left\{ \gamma \int_0^t{\left( -\nabla \cdot \vec{u} \left( \vec{X}\left( x,s \right) ,s \right) ds \right)} \right\} ,
	\end{split}
\end{equation}
and the uniqueness of particle paths, we have
	\begin{equation} \label{3.3}
		\rho = 0, P=0 \text{ on } [0, R(t)], \quad R(0) = r_0.
	\end{equation}
Consequently, we also have
\begin{equation}\label{3.4}
	P_r=0 \text{ on } [0, R(t)].
\end{equation}
	Then, in the vacuum region $[0, R(t)]$, the momentum equation in $\eqref{11.16}_2$ becomes
	\begin{equation} \label{3.5}
		(2\mu + \lambda) \left( u_r + \frac{u}{r} \right)_r = (2\mu + \lambda) \partial_r \nabla \cdot \vec{u} = B \left( B_r + \frac{B}{r} \right) \quad \text{a.e.}
	\end{equation}
	\subsection{Conservation law for the magnetic field}
	To deal with the equation \eqref{3.5}, a crucial observation is the conservation of $B$ within the vacuum region, specifically
	\begin{equation}\label{3.6}
		 \left| \int_0^{R(t)} B dr \right| = \left| \int_0^{R(0)=r_0} B_0 dr \right| = C_0 > 0.
	 \end{equation}
To verify this, we recall the regularity of $\vec{u}$ and $\vec{B}$ given 
\begin{equation*}
	\vec{u} \in C([0,T]; H^2) \subset C([0,T] \times \Omega), \quad \vec{B} \in C([0,T]; W^{1,q}) \subset C([0,T] \times \Omega),
\end{equation*}
which implies
\begin{equation*}
	u \in C([0,T] \times \Omega), \quad B \in C([0,T] \times \Omega).
\end{equation*}
Integrating the fourth equation of \eqref{11.16} over the interval $[0, R(t)]$ yields
\begin{equation}\label{3.7}
	\int_0^{R(t)} B_t dr = \frac{d}{dt} \left( \int_0^{R(t)} B dr \right) - R'(t) B(R(t), t).
\end{equation}
We note that the boundary $R(t)$ satisfies $R'(t) = u(R(t), t), R(0) = r_0.$
Furthermore, we have the identity
\begin{equation} \label{3.8}
	\int_0^{R(t)} (uB)_r dr = u(R(t), t) B(R(t), t).
\end{equation}
Combining \eqref{3.7}-\eqref{3.8}, we obtain
\begin{equation} \label{eq:3.41}
	\frac{d}{dt} \left( \int_0^{R(t)} B dr \right) = 0,
\end{equation}
which confirms the conservation law \eqref{3.6}.
	\subsection{Proof of Theorem \ref{Thm 1.2}}
To establish an upper bound on the lifespan of strong solutions, we employ a fractional moment argument on \eqref{3.5}. This approach allows us to identify a critical balance between viscous stress and the magnetic Lorentz force within the vacuum disk.\par
	The main idea to prove the above claim is to find a suitable multiplier for \eqref{3.5}. Indeed, for any $1 < \alpha < 2$, multiplying \eqref{3.5} by $(R(t) r^{\alpha} - r^{\alpha+1})$ gives
	\begin{equation}
		(2\mu + \lambda) (R(t) r^{\alpha} - r^{\alpha+1}) \partial_r \nabla \cdot \vec{u} = (R(t) r^{\alpha} - r^{\alpha+1}) \left( \partial_r \frac{B^2}{2} + \frac{B^2}{r} \right) \quad \text{a.e.}
	\end{equation}
By integrating the above equation over $[0, R(t)]$, we obtain
	\begin{equation} \label{3.20}
		(2\mu + \lambda) \int_0^{R(t)} (R(t) r^{\alpha} - r^{\alpha+1}) \partial_r \nabla \cdot \vec{u} dr = \int_0^{R(t)} (R(t) r^{\alpha} - r^{\alpha+1}) \left( \partial_r \frac{B^2}{2} + \frac{B^2}{r} \right) dr.
	\end{equation}
On the one hand, integrating the left-hand side of \eqref{3.20} by parts, we arrive at 
\begin{equation} \label{3.12}
(2\mu +\lambda )\int_0^{R(t)}{(}R\left( t \right) r^{\alpha}-r^{\alpha +1})\partial _r \nabla \cdot \vec{u}dr=-(2\mu +\lambda )\int_0^{R(t)}{(}\alpha R\left( t \right) r^{\alpha -1}-(\alpha +1)r^{\alpha}) \nabla \cdot \vec{u}dr.
\end{equation}
To prove that \eqref{3.12} holds for the strong solution $(\vec{u}, \vec{B})$, it is necessary to verify the regularity of the integrand. Since $\vec{u} \in C([0,T]; H^2)$ and $\partial_r \nabla \cdot \vec{u} = \nabla (\nabla \cdot \vec{u}) \cdot \frac{\vec{x}}{r}$, this implies that
	\begin{equation} \label{eq:3.23}
		\| r^{\frac{1}{2}} \partial_r \nabla \cdot \vec{u} \|_{L^2(0,R(t))} \le C\| \nabla (\nabla \cdot \vec{u}) \|_{L^2(\Omega)} \le C.
	\end{equation}
We rewrite the integrand on the left-hand side of \eqref{3.12} as
	\begin{equation} \label{eq:3.24}
		(R(t) r^{\alpha} - r^{\alpha+1}) \partial_r \nabla \cdot \vec{u} = (R(t) r^{\alpha-\frac{1}{2}} - r^{\alpha+\frac{1}{2}}) (r^{\frac{1}{2}} \partial_r \nabla \cdot \vec{u}) \in L^{\infty}((0,T); L^2(0,R)).
	\end{equation}
	On the other hand, the regularity $\vec{u} \in C([0,T]; H^2)$ ensures that $\nabla \cdot \vec{u} \in C([0,T]; L^p)$ for any $p \in (2, \infty)$, which implies
	\begin{equation} \label{eq:3.25}
		\| r^{\frac{1}{p}} \nabla \cdot \vec{u} \|_{L^p(0,R(t))} \le C\| \nabla \cdot \vec{u} \|_{L^p(\Omega)}.
	\end{equation}
	Thereby, taking $p$ large enough such that $\frac{1}{p} < \alpha - 1$, the integrand on the right-hand side of \eqref{3.12} can be rewritten as
	\begin{equation}\label{33.16}
		\begin{split}
				(\alpha R(t) &r^{\alpha-1} - (\alpha+1)r^{\alpha}) \nabla \cdot \vec{u}
				\\ 
			&= (\alpha R(t) r^{\alpha-1-\frac{1}{p}} - (\alpha+1)r^{\alpha-\frac{1}{p}}) (r^{\frac{1}{p}} \nabla \cdot \vec{u}) \in L^{\infty}((0,T); L^p(0,R(t))).
		\end{split}
	\end{equation}
Consequently, equation \eqref{3.12} can be verified using a standard smooth approximation argument. Specifically, for any $\epsilon > 0$, we split the integral into $\int_0^{\epsilon} + \int_{\epsilon}^R$, and obtain \eqref{3.12} by taking the limit as $\epsilon \to 0$.
	
	Therefore, the left-hand side of \eqref{3.12} can be bounded by
	\begin{equation} \label{3.27}
		\begin{aligned}
		\text{LHS}=&\left( 2\mu +\lambda \right) \left( \alpha R(t)\int_0^{R(t)}{r^{\alpha -1}}|\nabla \cdot\vec{u}|dr+(\alpha +1)\int_0^{R(t)}{r^{\alpha}}|\nabla \cdot\vec{u}|dr \right) 
		\\
		\le& \left( 2\mu +\lambda \right) \alpha R(t)\parallel \nabla \cdot \vec{u}\parallel _{L^2}\left( \int_0^{R(t)}{|}r^{\alpha -2}|^2rdr \right) ^{\frac{1}{2}}
		\\
		&+\left( 2\mu +\lambda \right) (\alpha +1)\parallel \nabla \cdot \vec{u}\parallel _{L^2}\left( \int_0^{R(t)}{|}r^{\alpha -1}|^2rdr \right) ^{\frac{1}{2}}
		\\
		=&\left( 2\mu +\lambda \right) \alpha R(t)\parallel \nabla \cdot \vec{u}\parallel _{L^2}\frac{1}{\sqrt{2\alpha -2}}{R(t})^{\alpha -1}
		\\
		&+\left( 2\mu +\lambda \right) (\alpha +1)\parallel \nabla \cdot \vec{u}\parallel _{L^2}\frac{1}{\sqrt{2\alpha}}R(t)^{\alpha}
		\\
		=&\left( 2\mu +\lambda \right) \left[ \frac{\alpha}{\sqrt{2\alpha -2}}+\frac{\alpha +1}{\sqrt{2\alpha}} \right] R(t)^{\alpha}\parallel \nabla \cdot \vec{u}\parallel _{L^2}.
		\end{aligned}
	\end{equation}
We estimate the right-hand side of \eqref{3.12} in a similar manner
	\begin{equation} \label{33.19}
		\begin{aligned}
			\text{RHS} &= \int_0^{R(t)} (R(t) r^{\alpha} - r^{\alpha+1}) \left( \partial_r \frac{B^2}{2} + \frac{B^2}{r} \right) dr \\
			&= -\int_0^{R(t)} \frac{B^2}{2} (R(t) \alpha r^{\alpha-1} - (\alpha+1)r^{\alpha}) dr + \int_0^{R(t)} B^2 (R(t) r^{\alpha-1} - r^{\alpha}) dr \\
			&= \int_0^{R(t)} \left( 1 - \frac{\alpha}{2} \right) B^2 R(t) r^{\alpha-1} dr + \int_0^{R(t)} \left( \frac{\alpha+1}{2} - 1 \right) B^2 r^{\alpha} dr \\
			&= \frac{2-\alpha}{2} \int_0^{R(t)} B^2 R(t) r^{\alpha-1} dr + \frac{\alpha-1}{2} \int_0^{R(t)} B^2 r^{\alpha} dr \\
			&\ge \frac{2-\alpha}{2} \int_0^{R(t)} B^2 R(t) r^{\alpha-1} dr. 
		\end{aligned}
	\end{equation}
	The identity \eqref{33.19} remains valid for the strong solution $\vec{B}$ via the same approximation procedure. The key step involves establishing the following regularity
	\begin{equation} \label{eq:3.30}
		(R(t) r^{\alpha} - r^{\alpha+1}) \partial_r \frac{B^2}{2} = (R(t) r^{\alpha-\frac{1}{q}} - r^{\alpha+1-\frac{1}{q}}) (r^{\frac{1}{q}} \partial_r \frac{B^2}{2}) \in L^{\infty}(0,T; L^q(0,R(t))).
	\end{equation}
To prove \eqref{eq:3.30}, we use the embedding $\vec{B} \in C([0,T]; W^{1,q}) \hookrightarrow C([0,T] \times \Omega)$, together with $B = \vec{B} \cdot \frac{x^{\perp}}{r}$ and $| \nabla \vec{B}|^2=B_{r}^{2}+\frac{B^2}{r^2}$. Consequently, we obtain
	\begin{equation} \label{eq:3.32}
		\| r^{\frac{1}{q}} \partial_r \frac{B^2}{2} \|_{L^q(0,R(t))} \le \| \vec{B} \|_{L^{\infty}(\Omega)} \| \nabla \vec{B} \|_{L^q(\Omega)} \le C.
	\end{equation}
and
	\begin{equation}
		B^2 r^{\alpha-1} + B^2 r^{\alpha} \le C \in L^{\infty}((0,T) \times \Omega).
	\end{equation}
This verifies the validity of the integration by parts performed in \eqref{33.19}.
\par	
	Substituting \eqref{3.27} and \eqref{33.19} into \eqref{3.20}, we obtain
	\begin{equation} \label{eq:3.34}
		\frac{2-\alpha}{2} \int_0^{R(t)} B^2 R(t) r^{\alpha-1} dr \le (2\mu + \lambda) \left[ \frac{\alpha}{\sqrt{2\alpha-2}} + \frac{\alpha+1}{\sqrt{2\alpha}} \right] R(t)^{\alpha} \| \nabla \cdot \vec{u}\|_{L^2}.
	\end{equation}
	Since we have the following inequality
	\begin{equation} \label{eq:3.42}
		\begin{aligned}
			C_0^2 = \left| \int_0^{R(t)} B dr \right|^2 &\le \int_0^{R(t)} B^2 r^{\alpha-1} dr \int_0^{R(t)} r^{1-\alpha} dr \\
			&= \int_0^{R(t)} B^2 r^{\alpha-1} dr \frac{1}{2-\alpha} R(t)^{2-\alpha},
		\end{aligned}
	\end{equation}
due to the condition $1 < \alpha < 2$.
\par	
Combining \eqref{eq:3.34} with \eqref{eq:3.42}, we deduce that
	\begin{equation} \label{eq:3.43}
		\begin{aligned}
			(2\mu + \lambda) \left[ \frac{\alpha}{\sqrt{2\alpha-2}} + \frac{\alpha+1}{\sqrt{2\alpha}} \right] R(t)^{\alpha} \|\nabla \cdot \vec{u}\|_{L^2} &\ge \frac{2-\alpha}{2} R(t) \int_0^{R(t)} B^2 r^{\alpha-1} dr \\
			&\ge \frac{2-\alpha}{2} R(t) C_0^2 (2-\alpha) R(t)^{\alpha-2} \\
			&= \frac{(2-\alpha)^2}{2} C_0^2 R(t)^{\alpha-1},
		\end{aligned}
	\end{equation}
	which implies
	\begin{equation} \label{eq:3.44}
		(2\mu + \lambda)  \|\nabla \cdot \vec{u}\|_{L^2} \ge \frac{C_0^2 (2-\alpha)^2}{2R(t) \left[ \frac{\alpha}{\sqrt{2\alpha-2}} + \frac{\alpha+1}{\sqrt{2\alpha}} \right]} > 0.
	\end{equation}
	Utilizing the fact that $R(t) \le R_0$, we obtain a lower estimate for $\|\nabla \cdot \vec{u}\|_{L^2}$
	\begin{equation} \label{eq:3.45}
		\| \nabla \cdot \vec{u}\|_{L^2} \ge \frac{1}{(2\mu + \lambda) R_0} \frac{C_0^2 (2-\alpha)^2}{2 \left[ \frac{\alpha}{\sqrt{2\alpha-2}} + \frac{\alpha+1}{\sqrt{2\alpha}} \right]} > 0.
	\end{equation}
Recalling the a priori energy estimates
\begin{equation}
	\frac{1}{2}\left( \| \sqrt{\rho}\vec{u} \| _{L^{\infty}L^2}^{2}+\| \vec{B} \| _{L^{\infty}L^2}^{2} \right) +\| \frac{P}{\gamma -1} \| _{L^{\infty}L^1}+(2\mu +\lambda )\int_0^T{\| \nabla \cdot \vec{u}\| _{L^2}^{2}dt}\le E_0.
\end{equation}
	here $E_0$ is the initial energy. Thus, the lifespan of the strong solution must satisfy
	\begin{equation}
		T \le \left( \frac{1}{\sqrt{2\mu + \lambda} R_0} \frac{C_0^2 (2-\alpha)^2}{2 \left[ \frac{\alpha}{\sqrt{2\alpha-2}} + \frac{\alpha+1}{\sqrt{2\alpha}} \right]} \right)^{-2} E_0.
	\end{equation}
	This finishes the proof of Theorem \ref{Thm 1.2}.
	\subsection{Proof of Theorem \ref{th1.3}}
	\begin{proof}
	The proof for the three-dimensional cylindrically symmetric case is similar to that of the two-dimensional radially symmetric case. Here we focus only on the behavior of equation \eqref{12.16}$_2$ within the interior vacuum region. Since we similarly obtain $P_r=0$ a.e. on $[0, R(t)]$, we have the following on $[0, R(t)]$
	
	\begin{equation}
		(2\mu +\lambda )\left( u_r+\frac{u}{r} \right) _r=B\left( B_r+\frac{B}{r} \right) \quad \mathrm{a}.\mathrm{e}.
	\end{equation}
We multiply the equation by the same multiplier $(R(t) r^{\alpha} - r^{\alpha+1})$ and perform the subsequent procedure. However, it should be noted that $\vec{u} \in C([0,T]; H^2)$ ensures that $\nabla \cdot \vec{u} \in C([0,T]; L^6)$ in the 3D case.
Therefore, we only have 
\begin{equation}
	\| r^{\frac{1}{6}} \nabla \cdot \vec{u} \|_{L^6(0,R(t))} \le C\| \nabla \cdot \vec{u} \|_{L^6(\Omega)}.
\end{equation}
Compared with \eqref{33.16}, to obtain
\begin{equation}
		\begin{split}
		(\alpha R(t) &r^{\alpha-1} - (\alpha+1)r^{\alpha}) \nabla \cdot \vec{u}
		\\ 
		&= (\alpha R(t) r^{\alpha-1-\frac{1}{6}} - (\alpha+1)r^{\alpha-\frac{1}{6}}) (r^{\frac{1}{6}} \nabla \cdot \vec{u}) \in L^{\infty}((0,T); L^6(0,R(t))),
	\end{split}
\end{equation}
it is necessary to assume $\alpha \ge \frac{7}{6}$.
	Similarly, applying $R(t) \le R_0$ yields a lower bound for $\|\nabla \cdot \vec{u}\|_{L^2}$
\begin{equation} \label{eq:33.45}
	\| \nabla \cdot \vec{u}\|_{L^2} \ge \frac{1}{(2\mu + \lambda) R_0} \frac{C_0^2 (2-\alpha)^2}{2 \left[ \frac{\alpha}{\sqrt{2\alpha-2}} + \frac{\alpha+1}{\sqrt{2\alpha}} \right]} > 0.
\end{equation}
Since we have $2\left( u_{r}^{2}+\frac{u^2}{r^2} \right) \ge \left( u_r+\frac{u}{r} \right) ^2=\left| \nabla \cdot \vec{u} \right|^2$ and the basic energy inequality
	\begin{equation}
\int_0^T{\int_0^{R_0}{[}}(2\mu +\lambda )(ru_{r}^{2}+\frac{u^2}{r})+\mu (rv_{r}^{2}+\frac{v^2}{r})+\mu rw_{r}^{2}]drdt\le E_0,
	\end{equation}
 the lifespan of the strong solution is consequently bounded by
\begin{equation}
	T \le \left( \frac{1}{\sqrt{2\mu + \lambda} R_0} \frac{C_0^2 (2-\alpha)^2}{2\sqrt{2} \left[ \frac{\alpha}{\sqrt{2\alpha-2}} + \frac{\alpha+1}{\sqrt{2\alpha}} \right]} \right)^{-2} E_0,
\end{equation}
for $\alpha \in [\frac{7}{6}, 2)$.
\end{proof}
\section{Finite time Blow-up for free boundary problem}
\subsection{Basic energy estimates}
Before proving the blowup result, we first derive the basic energy inequality under free boundary conditions. Here, our proof framework is formulated in radially symmetric coordinates.
\begin{lemma}\label{lemma 4.1}
	 Assume that $(\rho,\vec{u},P,\vec{B})$ is a symmetric strong solution of the \eqref{11.16} with the initial data \eqref{11.19}  and the boundary condition \eqref{11.25}. Then it holds that \textup{(}Neglecting the constant factor $2\pi$
\textup{)}
\begin{equation}\label{44.1}
	\begin{split}
		E_0&=\int_0^{a(0)}{\left( \frac{1}{2}\rho _0u_{0}^{2}r+\frac{1}{\gamma -1}P_0r+\frac{1}{2}B_{0}^{2}r \right)}dr
		\\
	&=\int_0^{a(t)}{\left( \frac{1}{2}\rho u^2r+\frac{1}{\gamma -1}Pr+\frac{1}{2}B^2r \right)}dr+(2\mu +\lambda )\int_0^t{\int_0^{a(t)}{\left( u_r+\frac{u}{r} \right) ^2}drd\tau}
	\\
	&=\frac{1}{2}\int_{\Omega _t}{\left( \rho |\vec{u}|^2+|\vec{B}|^2 \right)}dx+\int_{\Omega _t}{\frac{P}{\gamma -1}}dx+(2\mu +\lambda )\int_0^t{\left\| \nabla \cdot \vec{u} \right\| _{L^2}^{2}d\tau}.
	\end{split}
\end{equation}
\end{lemma}
\begin{proof}
Multiplying $\eqref{11.16}_2$ by $ur$ and integrating by parts, we get
\begin{equation}\label{4.1}
	\begin{split}
		\frac{1}{2}&\frac{d}{dt}\int_0^{a(t)}{\rho}u^2rdr-\int_0^{a(t)}{P}\left( u_r+\frac{u}{r} \right) rdr+(2\mu +\lambda )\int_0^{a(t)}{\left( u_r+\frac{u}{r} \right) ^2}dr
		\\
		&-\left( 2\mu +\lambda \right) \left( uu_rr+u^2\right) \left( a\left( t \right) ,t \right) +\left( Pur \right) \left( a\left( t \right) ,t \right) +\int_0^{a(t)}{(}uB^2+uBB_rr)dr=0.
	\end{split}
\end{equation}
multiplying $\eqref{11.16}_3$ by $r$ and integrating by parts, we arrive at
\begin{equation}\label{4.2}
	\frac{d}{dt}\int_0^{a\left( t \right)}{Prdr}-\int_0^{a\left( t \right)}{Pu_rrdr}-\int_0^{a\left( t \right)}{Pudr}+\gamma \int_0^{a\left( t \right)}{\left( Pu_rr+Pu \right) dr}=0.
\end{equation}
Multiplying $\eqref{11.16}_4$ by $Br$ and integrating by parts yields
\begin{equation}\label{4.3}
	\begin{aligned}
	\int_0^{a(t)}{B_t}Brdr+\int_0^{a(t)}{(}uB)_rBrdr=&\frac{1}{2}\frac{d}{dt}\int_0^{a(t)}{B^2}rdr-\frac{1}{2}uB^2(t,a(t))a(t)
	\\
	&+\left( uB^2 \right) \left( t,a(t) \right) a(t)-\int_0^{a(t)}{(}uB^2+uBB_rr)dr
	\\
	=&\frac{1}{2}\frac{d}{dt}\int_0^{a(t)}{B^2}rdr+\frac{1}{2}\left( uB^2r \right) \left( a\left( t \right) ,t \right) 
	\\
	&-\int_0^{a(t)}{(}uB^2+uBB_rr)dr=0.
	\end{aligned}
\end{equation}
In deriving \eqref{4.1}--\eqref{4.3}, we applied the condition \eqref{11.25} and \eqref{11.26}. 
Now adding \eqref{4.1} to \eqref{4.2} multiplied by $\frac{1}{\gamma-1}$, and  adding \eqref{4.3} yields
\begin{equation}
	\frac{d}{dt}\int_0^{a(t)}{\left( \frac{1}{2}\rho u^2r+\frac{1}{\gamma -1}Pr+\frac{1}{2}B^2r \right)}dr+(2\mu +\lambda )\int_0^{a(t)}{\left( u_r+\frac{u}{r} \right) ^2}dr=0,
\end{equation}
where we have used
\begin{equation}
	\left( \frac{B^2}{2}+P-\left( 2\mu +\lambda \right) \left( u_r+\frac{u}{r} \right) \right) \left( a\left( t \right) ,t \right) =0.
\end{equation} 	
This concludes the proof of Lemma \ref{lemma 4.1}.
\end{proof}
\subsection{Estimates for the free boundary radius}

In this subsection, we derive a crucial upper bound for the growth of the free boundary radius $a(t)$, following \cite{Li-Zhang}.  This estimate provides control over the expansion of the fluid domain and plays a pivotal role in the subsequent blow-up analysis.

\begin{lemma}\label{lem:boundary_estimate}
	Under the assumptions of Theorem \ref{thm 1.4}, the free boundary $a(t)$ satisfies the following growth estimate
	\begin{equation}\label{eq:boundary_growth}
		a(t) \le C(1+t)^{\frac{1}{2}}, \quad \text{for all } t \ge 0,
	\end{equation}
	where $C$ is a positive constant depending on $a_0,E_0,\mu,\lambda$.
\end{lemma}

\begin{proof}
	Recall that the evolution of the free boundary is governed by the condition $a'(t) = u(a(t), t)$. Integrating this relation with respect to time from $0$ to $t$, we can express the boundary radius as
	\begin{equation}
		a(t) = a_0 + \int_0^t a'(s) \, ds = a_0 + \int_0^t u(a(s), s) \, ds.
	\end{equation}
	To estimate the integral term, we apply the Cauchy-Schwarz inequality
	\begin{equation}
		\int_0^t |u(a(s), s)| \, ds \le \left( \int_0^t 1 \, ds \right)^{\frac{1}{2}} \left( \int_0^t |u(a(s), s)|^2 \, ds \right)^{\frac{1}{2}} = t^{\frac{1}{2}} \left( \int_0^t u^2(a(s), s) \, ds \right)^{\frac{1}{2}}.
	\end{equation}
	Next, we estimate the boundary velocity term $\int_0^t u^2(a(s), s) \, ds$. Using the boundary condition $u(0, t) = 0$, we have
	\begin{equation}
		u^2(a(s), s) = \int_0^{a(s)} \partial_r (u^2(r, s)) \, dr = \int_0^{a(s)} 2 u(r, s) u_r(r, s) \, dr.
	\end{equation}

The boundary term is bounded directly by the viscous dissipation term appearing in the energy identity. Specifically
	\begin{equation}
		\int_0^t u^2(a(s), s) \, ds \le \int_0^t \int_0^{a(s)} 2 u u_r \, dr ds \le \int_0^t \int_0^{a(s)} \left( u_r + \frac{u}{r} \right)^2 r \, dr ds.
	\end{equation}
	From the global energy identity \eqref{44.1}, we know that the total viscous dissipation is bounded by the initial energy $E_0$
	\begin{equation}
		(2\mu + \lambda) \int_0^t \int_0^{a(s)} \left( u_r + \frac{u}{r} \right)^2 r \, dr ds \le E_0.
	\end{equation}
	Therefore, we have
	\begin{equation}
		\int_0^t{u^2}(a(s),s)\,ds\le \frac{E_0}{2\mu +\lambda}.
	\end{equation}
	Combining these estimates, we arrive at
	\begin{equation}
	a(t)\le a_0+t^{\frac{1}{2}}\left( \frac{E_0}{2\mu +\lambda} \right) ^{\frac{1}{2}}\le C(1+t)^{\frac{1}{2}},
	\end{equation}
	where $C$ is dependent on $a_0,E_0,\mu,\lambda$. This completes the proof of Lemma \ref{lem:boundary_estimate}.
\end{proof}
\subsection{Proof of Theorem \ref{thm 1.4}}
\begin{proof}
	In this case, we are dealing with two distinct moving interfaces: the inner vacuum boundary, denoted by $R(t)$; and the outer free surface, denoted by $a(t)$. The condition $\nabla \vec{u} \in L^2(0,T; L^{\infty})$ guarantees the uniqueness and smoothness of the particle trajectories, which implies that particle paths cannot intersect. Consequently, the inner vacuum boundary can never overtake the outer free surface. Therefore, we have
	\begin{equation}
		R(t) \le a(t), \quad \text{for all } t \in [0, T).
	\end{equation}
	Combining this with the growth estimate for the free boundary radius derived in Lemma \ref{lem:boundary_estimate}, we obtain a uniform upper bound for the vacuum radius
	\begin{equation} \label{eq:R_bound}
		R(t) \le a(t) \le C(1+t)^{\frac{1}{2}}.
	\end{equation}
	
	Next, we analyze the dynamics within the vacuum region. As established in the fixed boundary case, the density and pressure vanish in $[0, R(t)]$, reducing the momentum equation to a balance between viscous stress and the Lorentz force
	\begin{equation}
		(2\mu + \lambda) \partial_r \nabla \cdot \vec{u} = B \left( B_r + \frac{B}{r} \right) \quad \text{a.e. in } (0, R(t)).
	\end{equation}
	Notice that the interior vacuum boundary $R(t)$ behaves essentially the same way as in the fixed boundary case discussed in Section \ref{s4}. Therefore, we can employ the identical fractional moment test function, $f(r) = R(t) r^{\alpha} - r^{\alpha+1}$ with $1 < \alpha < 2$. By repeating the integration and estimation procedures over the interval $[0, R(t)]$, we get the lower bound for the $\|\nabla \cdot \vec{u}\|_{L^2}$
	\begin{equation}
		\|\nabla \cdot \vec{u}\|_{L^2} \ge \frac{1}{(2\mu+\lambda)R(t)} \frac{C_0^2(2-\alpha)^2}{2\left[\frac{\alpha}{\sqrt{2\alpha-2}} + \frac{\alpha+1}{\sqrt{2\alpha}}\right]},
	\end{equation}
	where the constant $C_0$ represents the conserved total magnetic flux in the vacuum region, as defined in \eqref{3.6}.
	
	A critical step in the free boundary analysis is to relate this lower bound to the time variable $t$. Utilizing the geometric constraint \eqref{eq:R_bound}, we can replace $R(t)$ with its upper bound $C(1+t)^{1/2}$. This yields the following time-dependent lower estimate for the viscous dissipation
	\begin{equation} \label{eq:div_lower_bound}
		\begin{aligned}
			\|\nabla \cdot \vec{u}\|_{L^2} &\ge \frac{1}{(2\mu+\lambda)a(t)} \frac{C_0^2(2-\alpha)^2}{2\left[\frac{\alpha}{\sqrt{2\alpha-2}} + \frac{\alpha+1}{\sqrt{2\alpha}}\right]} \\
			&\ge \frac{1}{(2\mu+\lambda)C(1+t)^{\frac{1}{2}}} \frac{C_0^2(2-\alpha)^2}{2\left[\frac{\alpha}{\sqrt{2\alpha-2}} + \frac{\alpha+1}{\sqrt{2\alpha}}\right]}.
		\end{aligned}
	\end{equation}
	 We recall the global energy conservation law \eqref{44.1}, which implies the integrability of the viscous dissipation over time
	\begin{equation}
		(2\mu + \lambda) \int_0^T \|\nabla \cdot \vec{u} \|_{L^2}^2 \, dt \le E_0.
	\end{equation}
	Substituting the lower bound \eqref{eq:div_lower_bound} into this energy inequality, we obtain
	\begin{equation}
		(2\mu +\lambda)\left( \frac{C_{0}^{2}(2-\alpha )^2}{2(2\mu +\lambda )C \left[ \frac{\alpha}{\sqrt{2\alpha -2}}+\frac{\alpha +1}{\sqrt{2\alpha}} \right]} \right) ^2 \int_0^T \frac{1}{1+t} \, dt \le E_0.
	\end{equation}
	Computing the time integral yields a logarithmic term $\log(1+T)$. Rearranging the inequality provides an explicit upper bound for the lifespan $T$ of the strong solution
	\begin{equation}
		\log\mathrm{(}1+T)\le E_0\left( \frac{C_{0}^{2}(2-\alpha )^2}{2\sqrt{2\mu +\lambda}C\left[ \frac{\alpha}{\sqrt{2\alpha -2}}+\frac{\alpha +1}{\sqrt{2\alpha}} \right]} \right) ^{-2}.
	\end{equation}
We conclude that the maximal existence time is finite
	\begin{equation}
T\le \exp \left( E_0\left( \frac{C_{0}^{2}(2-\alpha )^2}{2\sqrt{2\mu +\lambda}C\left[ \frac{\alpha}{\sqrt{2\alpha -2}}+\frac{\alpha +1}{\sqrt{2\alpha}} \right]} \right) ^{-2} \right) -1.
	\end{equation}
	This confirms that the strong solution must blow up in finite time.
\end{proof}
 	\section{Local well-posedness}\label{s2}
 	In this section, We provide the proof of the local well-posedness stated in Theorem \ref{th1}.
\subsection{Linearized problem}
We assume that $\Omega$ is a bounded domain in $\mathbb{R} ^2$ with smooth boundary and first prove the local existence of strong solutions with positive densities to the linearized equation. Furthermore, we derive uniform estimates that are independent of the initial density's lower bound. These crucial estimates will be applied in the proof of the existence of strong solutions with nonnegative densities.
Consider the following linearized problem for 
\begin{equation}\label{2.1}
	\begin{cases}
		\rho _t+\nabla \cdot \left( \rho v \right) =0,\\
		\left( \rho u \right) _t+\nabla \cdot \left( \rho v\otimes u \right) -\mu \Delta u-\left( \mu +\lambda \right) \nabla \left( \nabla \cdot u \right) +\nabla P=\left( \nabla \times B \right) \times B,\\
		P_t+v\cdot \nabla P+\gamma P\nabla \cdot v=0,\\
		B_t-\nabla \times \left( v\times B \right) =0,\\
		\nabla \cdot B=0,\\
	\end{cases}
\end{equation}
with the initial boundary conditions
\begin{equation}\label{2.2}
	\begin{cases}
		\left( \rho , \rho u, P, B \right) |_{t=0}=\left( \rho _0, \rho _0u_0, P_0, B_0 \right) ,\\
		u=0 \,\,\,\mathrm{on} \,\partial \Omega .\\
	\end{cases}
\end{equation}
Here $\rho_0, u_0, P_0, B_0$ satisfy the following regularity conditions
\begin{equation}\label{2.3}
	\begin{cases}
		0\le \rho _0,P_0\in W^{1,q}, B_0\in W^{1,q}, \nabla \cdot B_0=0 \,\mathrm{in} \,\Omega ,\\
		u_0\in H^2, u_0|_{\partial \Omega}=0,\\
	\end{cases}
\end{equation}
for some $q>2$ and the additional conditions
\begin{equation}\label{2.44}
	\begin{cases}
		\rho _{0}^{1/2}g=-\mu \Delta u_0-\left( \mu +\lambda \right) \nabla \left( \nabla \cdot u_0 \right) +\nabla P_0+B_0\times \left( \nabla \times B_0 \right) , \mathrm{for}\, \mathrm{some}\, g\in L^2,\\
		\rho _0\ge \delta >0,\\
		2+\left\| \left( \rho _0,B_0,P_0 \right) \right\| _{W^{1,q}}+\left\| u_0 \right\| _{H^2}+\left\| g \right\| _{L^2}^{2}<c_0.\\
	\end{cases}
\end{equation}
The given function $v$ satisfies
\begin{equation}\label{2.55}
	v(0)=u_0, v|_{\partial \Omega}=0,
\end{equation}
and the following regularity conditions
\begin{equation}\label{2.5}
	v\in C\left( \left[ 0,T^* \right] ;H^2 \right) \cap L^2\left( 0,T^*;W^{2,q} \right) ,v_t\in L^2\left( 0,T^*;H^1 \right).
\end{equation}
Moreover, $v$ is required to meet the following bounds 
\begin{equation}\label{2.6}
	\left\| v \right\| _{L^{\infty}\left( 0,T^*;H^1 \right)}+\kappa ^{-1}\left\| v \right\| _{L^{\infty}\left( 0,T^*;H^2 \right)}+\left\| v_t \right\| _{L^2\left( 0,T^*;H^1 \right)}+\left\| v \right\| _{{L^2}\left( 0,T^*;W^{2,q} \right)}\le c_1.
\end{equation}
The above fixed positive constants $c_0$, $c_1$, $\kappa$, and $T^*$ satisfy
\begin{equation}
	1<c_0<c_1<c_2\overset{\mathrm{def}}{=} \kappa c_1,\, \mathrm{and}\, 0<T^*<\infty .
\end{equation}
We begin by establishing an existence result for the case of positive initial densities.
\begin{lemma}\label{lemma 2.1}
	Let $\Omega$ be a bounded domain in $\mathbb{R} ^2$ with a smooth boundary. Assume that the initial data $(\rho_{0}, u_0, B_0, P_0)$ satisfy the regularity condition \eqref{2.3} and \eqref{2.44}. Suppose further that the function $v$ satisfies condition \eqref{2.55}, \eqref{2.5} and \eqref{2.6}.
	Then the initial-boundary value problem \eqref{2.1}, \eqref{2.2} admits a unique strong solution with the property that
	\begin{equation}\label{2.8}
		\begin{cases}
			\left( \rho ,P,B \right) \in C\left( \left[ 0,T_{*} \right] ; W^{1,q} \right) , \left( \rho _t,P_t,B_t \right) \in C\left( \left[ 0,T_{*} \right] ; L^q \right) ,\,\,\\
			u\in C\left( \left[ 0,T_{*} \right] ;H^2 \right) \cap L^2\left( 0,T_{*};W^{2,q} \right) , u_t\in L^2\left( 0,T_{*};H^1 \right) ,\\
			\sqrt{\rho}u_t\in L^{\infty}\left( 0,T_{*};L^2 \right) ,\\
		\end{cases}
	\end{equation}
	where $T_*=T_3$ is defined in \eqref{2.42}.
	
	Moreover, there exists a constant $C$ such that
	\begin{equation}
		\begin{cases}\label{2.100}
			\left\| \rho \right\| _{L^{\infty}W^{1,q}}+\left\| \rho _t \right\| _{L^{\infty}L^q}+\left\| P \right\| _{L^{\infty}W^{1,q}}+\left\| P_t \right\| _{L^{\infty}L^q}+\left\| B \right\| _{L^{\infty}W^{1,q}}+\left\| B_t \right\| _{L^{\infty}L^q}\le Cc^2_2,\\
			\left\| \sqrt{\rho}u_t \right\| _{L^{\infty}L^2}\le Cc_{0}^{\frac{5}{2}},\\
			\left\| u \right\| _{L^{\infty}H^1}+\kappa ^{-1}\left\| u \right\| _{L^{\infty}H^2}+\int_0^{T_*}{\left\| u_t\left( s \right) \right\| _{H^1}^{2}+\left\| u\left( s \right) \right\| _{W^{2,q}}^{2}ds}\le c_1,\\
		\end{cases}
	\end{equation}
	provided 
	\begin{equation}\label{2.111}
		c_1\overset{\mathrm{def}}{=}3Cc_{0}^{7},c_2\overset{\mathrm{def}}{=}27C^4c_{0}^{\frac{37}{2}},\kappa \overset{\mathrm{def}}{=}9C^3c_{0}^{\frac{23}{2}}.
	\end{equation}
\end{lemma}
\begin{proof}
	The existence of a unique solution $\rho$ was demonstrated by Diperna and Lions \cite{DiPerna-Lions}, a result also noted in Lemma 3 of \cite{Cho-Choe-Kim}. It follows from (2.11) in \cite{Cho-Choe-Kim} that 
	\begin{equation}
		\left\| \rho \left( t \right) \right\| _{W^{1,q}}\le \left\| \rho _0 \right\| _{W^{1,q}}\exp \left( C\int_0^t{\left\| \nabla v\left( s \right) \right\| _{W^{1,q}}ds} \right) .
	\end{equation}
	Note that
	\begin{equation}\label{2.12}
		\int_0^{T_1}{\left\| \nabla v\left( s \right) \right\| _{W^{1,q}}ds}\le T_{1}^{\frac{1}{2}}\left( \int_0^{T_1}{\left\| \nabla v\left( s \right) \right\| _{W^{1,q}}^{2}ds} \right) ^{\frac{1}{2}}\le Cc_1T_{1}^{\frac{1}{2}}\le C,
	\end{equation}
	provided 
	\begin{equation}
		T_1\overset{\mathrm{def}}{=}c_{1}^{-2}<1.
	\end{equation}
	We concluded that for $0\le t\le \min(T_1,T^*)$, 
	$$\left\| \rho \left( t \right) \right\| _{W^{1,q}}\le Cc_0,$$
	and 
	\begin{equation*}
		\begin{split}
			\left\| \rho _t\left( t \right) \right\| _{L^q}&=\left\| \nabla \cdot \left( \rho v \right) \right\| _{L^q}\le C\left( \left\| \rho \right\| _{W^{1,q}}\left\| v \right\| _{W^{1,q}} \right) 
			\\
			&\le Cc_0\left\| v \right\| _{H^2}\le Cc_0c_2.
		\end{split}
	\end{equation*}
	As a consequence of the transport equation representation, we have 
	\begin{equation}
		\rho \left( t,x \right) \ge \delta \exp \left[ -\int_0^t{\left\| \nabla v\left( s \right) \right\| _{L^{\infty}}ds} \right].
	\end{equation}
	Therefore, for $0\le t\le \min(T_1,T^*)$, it holds that 
	$$C^{-1}\delta \le \rho \left( t,x \right) \le Cc_0.$$
	\par
	The next step is to derive the regularity of $B$ for the linearized system \eqref{2.1}.
	Since 
	\begin{equation}\label{2.18}
		B_t+\left( v\cdot \nabla \right) B-\left( B\cdot \nabla \right) v+B\nabla \cdot v=0,
	\end{equation}
	multiplying by $\left| B \right|^{q-2}B$ and integrating by parts, we obtain 
	\begin{equation}\label{2.19}
		\begin{split}
			\frac{d}{dt}\int{\left| B \right|^qdx}&\le C\int{\left| \nabla v \right|\left| B \right|^qdx}\le C\left\| \nabla v \right\| _{L^{\infty}}\left\| B \right\| _{L^q}^{q}.
		\end{split}
	\end{equation}
	Then differentiating \eqref{2.18} with respect to $x_j$, multiplying by $\partial _jB_i\left| \partial _jB_i \right|^{q-2}$ and integrating by parts, we have
	\begin{equation}\label{2.20}
		\begin{split}
			\frac{d}{dt}\int{\left| \partial _jB_i \right|^qdx}&\le C\int{\left| \nabla v \right|\left| \nabla B \right|^q+\left| B \right|\left| \nabla B \right|^{q-1}\left| \nabla ^2v \right|dx}
			\\
			&\le C\left\| \nabla v \right\| _{L^{\infty}}\left\| \nabla B \right\| _{L^q}^{q}+C\left\| \nabla ^2v \right\| _{L^q}\left\| \nabla B \right\| _{L^q}^{q-1}\left\| B \right\| _{L^{\infty}}.
		\end{split}
	\end{equation}
	Combining \eqref{2.19} and \eqref{2.20}, and using Sobolev inequality yields
	\begin{equation}\label{2.21}
		\frac{d}{dt}\left\| B \right\| _{W^{1,q}}\le C\left\| \nabla v \right\| _{W^{1,q}}\left\| B \right\| _{W^{1,q}}.
	\end{equation}
	An application of Gronwall's inequality and \eqref{2.12} implies that, for $0\le t\le \min(T_1,T^*)$, 
	$$\left\| B \left( t \right) \right\| _{W^{1,q}}\le Cc_0,$$
	and 
	\begin{equation*}
		\begin{split}
			\left\| B _t\left( t \right) \right\| _{L^q}&\le C\left( \left\| B \right\| _{W^{1,q}}\left\| v \right\| _{W^{1,q}} \right) 
			\\
			&\le Cc_0\left\| v \right\| _{H^2}\le Cc_0c_2.
		\end{split}
	\end{equation*}
	\par
	We now turn to the existence of the pressure $P$ and the corresponding estimates. The approach for estimating the pressure is similar to that used for the density and the magnetic field, since the pressure $P$ satisfies
	\begin{equation}\label{2.24}
		P_t+v\cdot \nabla P+\gamma P\nabla \cdot v=0.
	\end{equation}
	By applying the same method used for \eqref{2.20} and \eqref{2.21}, we can easily obtain
	\begin{equation}
		\left\| P\left( t \right) \right\| _{W^{1,q}}\le \left\| P_0 \right\| _{W^{1,q}}\exp \left( C\int_0^t{\left\| \nabla v\left( s \right) \right\| _{W^{1,q}}ds} \right) .
	\end{equation} 
	Therefore, for $0\le t\le \min(T_1,T^*)$, we get 
	\begin{equation*}
		\left\| P \left( t \right) \right\| _{W^{1,q}}\le Cc_0,
	\end{equation*}
	and 
	\begin{equation*}
		\begin{split}
			\left\| P_t\left( t \right) \right\| _{L^q}&=\left\| v\cdot \nabla P+\gamma P\nabla \cdot v \right\| _{L^q}\le C\left( \left\| P \right\| _{W^{1,q}}\left\| v \right\| _{W^{1,q}} \right) 
			\\
			&\le Cc_0\left\| v \right\| _{H^2}\le Cc_0c_2.
		\end{split}
	\end{equation*}
	\par
	Finally, we prove the existence of $u$ to the linearized system \eqref{2.1}. We differentiate $\eqref{2.1}_2$ with respect to $t$ and multiply the resulting equation by $u_{t}$, we arrive at
	\begin{equation}
		\begin{split}
			\frac{1}{2}&\frac{d}{dt}\int{\rho u_{t}^{2}dx}+\int{\mu \left| \nabla u_t \right|^2+\left( \mu +\lambda \right) \left| \nabla \cdot u_t \right|^2dx}
			\\
			&=\int{\left( -\nabla P_t-\rho _tv\cdot \nabla u-\rho \left( 2v\cdot \nabla u_t+v_t\cdot \nabla u \right) +\nabla \cdot \left( BB_tI-2B\otimes B_t \right) \right) u_tdx}.
		\end{split}
	\end{equation}
	A combined application of the H\"older, Young and Gagliardo-Nirenberg inequalities yields
	\begin{equation}\label{2.29}
		\begin{split}
			\frac{1}{2}\frac{d}{dt}&\int{\rho u_{t}^{2}dx}+\int{\mu \left| \nabla u_t \right|^2dx}
			\\
			\le &C\left( \left\| P_t \right\| _{L^2}^{2}+\left\| \rho _t \right\| _{L^2}^{2}\left\| v \right\| _{L^{\infty}}^{2}\left\| \nabla u \right\| _{L^4}^{2}+\left\| \sqrt{\rho} \right\| _{L^{\infty}}^{2}\left\| v \right\| _{L^{\infty}}^{2}\left\| \sqrt{\rho}u_t \right\| _{L^2}^{2} \right. 
			\\
			&\left. +\left\| B \right\| _{L^{\infty}}^{2}\left\| B_t \right\| _{L^2}^{2}+\left\| \sqrt{\rho} \right\| _{L^{\infty}}\left\| v_t \right\| _{L^4}\left\| \nabla u \right\| _{L^4}\left\| \sqrt{\rho}u_t \right\| _{L^2} \right) 
			\\
			\le& C\left( \left\| P_t \right\| _{L^2}^{2}+\left\| \rho _t \right\| _{L^2}^{2}\left\| v \right\| _{L^{\infty}}^{2}\left\| \nabla u \right\| _{L^4}^{2}+\left\| \sqrt{\rho} \right\| _{L^{\infty}}^{2}\left\| v \right\| _{L^{\infty}}^{2}\left\| \sqrt{\rho}u_t \right\| _{L^2}^{2}+\left\| B \right\| _{L^{\infty}}^{2}\left\| B_t \right\| _{L^2}^{2} \right) 
			\\
			&+\frac{C}{\varepsilon}\left\| \sqrt{\rho} \right\| _{L^{\infty}}^{2}\left\| \nabla u \right\| _{L^2}\left\| \nabla u \right\| _{H^1}+\varepsilon \left\| \nabla v_t \right\| _{L^2}^{2}\left\| \sqrt{\rho}u_t \right\| _{L^2}^{2},
		\end{split}
	\end{equation}
	for any $\varepsilon>0$. Using the fact that 
	\begin{equation}\label{2.30}
		\frac{1}{2}\frac{d}{dt}\int{\left| \nabla u \right|^2dx}=\int{\nabla u\cdot \nabla u_tdx}\le \frac{\mu}{2}\int{\left| \nabla u_t \right|^2dx}+C\int{\left| \nabla u \right|^2dx},
	\end{equation}
	and it follows from the elliptic regularity results that 
	\begin{equation}\label{2.31}
		\begin{split}
			\left\| \nabla u \right\| _{H^1}&\le C\left( \left\| \rho u_t \right\| _{L^2}+\left\| \rho v\cdot \nabla u \right\| _{L^2}+\left\| \nabla P \right\| _{L^2}+\left\| B \right\| _{W^{1,q}}\left\| \nabla B \right\| _{L^2} \right) 
			\\
			&\le C\left( c_{0}^{\frac{1}{2}}\left\| \sqrt{\rho}u_t \right\| _{L^2}+c_0c_2\left\| \nabla u \right\| _{L^2}+c_0+c_{0}^{2} \right),
		\end{split}
	\end{equation}
	so one can get
	\begin{equation}\label{2.32}
		\begin{split}
			\left\| \rho _t \right\| _{L^2}^{2}\left\| v \right\| _{L^{\infty}}^{2}\left\| \nabla u \right\| _{L^4}^{2}&\le Cc_{0}^{2}c_{2}^{4}\left\| \nabla u \right\| _{H^1}^{2}
			\\
			&\le Cc_{0}^{2}c_{2}^{4}\left( c_0\left\| \sqrt{\rho}u_t \right\| _{L^2}^{2}+c_{0}^{2}c_{2}^{2}\left\| \nabla u \right\| _{L^2}^{2}+c_{0}^{4} \right) 
			\\
			&\le Cc_{0}^{3}c_{2}^{4}\left\| \sqrt{\rho}u_t \right\| _{L^2}^{2}+Cc_{0}^{4}c_{2}^{6}\left\| \nabla u \right\| _{L^2}^{2}+Cc_{0}^{6}c_{2}^{4}.
		\end{split}
	\end{equation}
	Substituting \eqref{2.30}-\eqref{2.32} into \eqref{2.29} yields
	\begin{equation}
		\begin{split}
			\frac{d}{dt}&\int{\rho u_{t}^{2}+\left| \nabla u \right|^2dx}+\int{\mu \left| \nabla u_t \right|^2dx}
			\\
			\le& C\left( c_{0}^{2}c_{2}^{2}+c_{0}^{3}c_{2}^{4}\left\| \sqrt{\rho}u_t \right\| _{L^2}^{2}+c_{0}^{4}c_{2}^{6}\left\| \nabla u \right\| _{L^2}^{2}+c_{0}^{6}c_{2}^{4}+c_0c_{2}^{2}\left\| \sqrt{\rho}u_t \right\| _{L^2}^{2}+c_{0}^{4}c_{2}^{2} \right) 
			\\
			&+\frac{C}{\varepsilon}c_0\left\| \nabla u \right\| _{L^2}\left( c_{0}^{\frac{1}{2}}\left\| \sqrt{\rho}u_t \right\| _{L^2}+c_0c_2\left\| \nabla u \right\| _{L^2}+c_{0}^{2} \right) +\varepsilon \left\| \nabla v_t \right\| _{L^2}^{2}\left\| \sqrt{\rho}u_t \right\| _{L^2}^{2}
			\\
			\le& C\left( c_{0}^{3}c_{2}^{4}\left\| \sqrt{\rho}u_t \right\| _{L^2}^{2}+c_{0}^{4}c_{2}^{6}\left\| \nabla u \right\| _{L^2}^{2}+c_{0}^{6}c_{2}^{4} \right) 
			\\
			&+\frac{C}{\varepsilon}c_0\left( \left\| \nabla u \right\| _{L^2}^{2}+c_0\left\| \sqrt{\rho}u_t \right\| _{L^2}^{2}+c_{0}^{2}c_{2}^{2}\left\| \nabla u \right\| _{L^2}^{2}+c_{0}^{2} \right) +\varepsilon \left\| \nabla v_t \right\| _{L^2}^{2}\left\| \sqrt{\rho}u_t \right\| _{L^2}^{2}
			\\
			\le& C\left( c_{0}^{6}c_{2}^{4}+\frac{c_{0}^{3}}{\varepsilon} \right) +C\left( c_{0}^{3}c_{2}^{4}+\frac{c_{0}^{2}}{\varepsilon} \right) \left\| \sqrt{\rho}u_t \right\| _{L^2}^{2}
			\\
			&+C\left( c_{0}^{4}c_{2}^{6}+\frac{c_{0}^{3}c_{2}^{2}}{\varepsilon} \right) \left\| \nabla u \right\| _{L^2}^{2}+\varepsilon \left\| \nabla v_t \right\| _{L^2}^{2}\left\| \sqrt{\rho}u_t \right\| _{L^2}^{2}
			\\
			\le& \left( \frac{C}{\varepsilon}c_{2}^{10}+\varepsilon \left\| \nabla v_t \right\| _{L^2}^{2} \right) \left( 1+\left\| \sqrt{\rho}u_t \right\| _{L^2}^{2}+\left\| \nabla u \right\| _{L^2}^{2} \right) .
		\end{split}
	\end{equation}
	Using Gronwall's inequality, we get 
	\begin{equation}\label{2.34}
		\begin{split}
			&\underset{0\le t\le T_2}{\mathrm{sup}}\left( \left\| \sqrt{\rho}u_t \right\| _{L^2}^{2}+\left\| \nabla u \right\| _{L^2}^{2} \right) +\int_0^{T_2}{\left\| \nabla u_t \right\| _{L^2}^{2}dt}
			\\
			\le& \exp \left( C\int_0^{T_2}{\frac{1}{\varepsilon}c_{2}^{10}+\varepsilon \left\| \nabla v_t \right\| _{L^2}^{2}dt} \right) \left( \underset{t\rightarrow 0^+}{\lim}\left\| \sqrt{\rho}u_t \right\| _{L^2}^{2}+\left\| \nabla u_0 \right\| _{L^2}^{2}+\int_0^{T_2}{\frac{C}{\varepsilon}c_{2}^{10}+\varepsilon \left\| \nabla v_t \right\| _{L^2}^{2}dt} \right) .
		\end{split}
	\end{equation} 
	Since one can rewrite $\eqref{2.1}_2$ as 
	$$\rho ^{\frac{1}{2}}u_t=-\rho ^{\frac{1}{2}}v\cdot \nabla u-g.$$
	Therefore, it holds that 
	\begin{equation*}
		\begin{split}
			\underset{t\rightarrow 0^+}{\lim}\left\| \sqrt{\rho}u_t \right\|_{L^2} &=\underset{t\rightarrow 0^+}{\lim}\left\| \sqrt{\rho}v\cdot \nabla u+g \right\| _{L^2}
			\\
			&\le \left\| \rho _0 \right\| _{L^{\infty}}^{\frac{1}{2}}\left\| u_0 \right\| _{L^{\infty}}\left\| \nabla u_0 \right\| _{L^2}+\left\| g \right\| _{L^2}\le Cc_{0}^{\frac{5}{2}}.
		\end{split}
	\end{equation*}  
	Taking $\varepsilon =c_{1}^{-2},T_2=\min \left(T^*, T_1,\frac{1}{c_{1}^{2}c_{2}^{10}} \right) $, it follows from \eqref{2.34} that 
	\begin{equation*}
		\underset{0\le t\le T_2}{\mathrm{sup}}\left( \left\| \sqrt{\rho}u_t \right\| _{L^2}^{2}+\left\| \nabla u \right\| _{L^2}^{2} \right) +\int_0^{T_2}{\left\| \nabla u_t \right\| _{L^2}^{2}dt}\le Cc_{0}^{5}.
	\end{equation*}
	Since
	\begin{equation}
		\begin{split}
			\left\| \nabla u \right\| _{H^1}&\le C\left( c_{0}^{\frac{1}{2}}\left\| \sqrt{\rho}u_t \right\| _{L^2}+c_0\left\| v \right\| _{L^4}\left\| \nabla u \right\| _{L^4}+\left\| \nabla P \right\| _{L^2}+c_{0}^{2} \right) 
			\\
			&\le C\left( c_{0}^{3}+c_0c_1\left\| \nabla u \right\| _{L^2}^{\frac{1}{2}}\left\| \nabla u \right\| _{H^1}^{\frac{1}{2}} \right) ,
		\end{split}
	\end{equation}
	we have 
	\begin{equation*}
		\begin{split}
			\left\| \nabla u \right\| _{H^1}&\le C\left( c_{0}^{3}+c_{0}^{2}c_{1}^{2}\left\| \nabla u \right\| _{L^2} \right) 
			\\
			&\le Cc_{0}^{\frac{9}{2}}c_{1}^{2}.
		\end{split}
	\end{equation*}
	We then invoke the standard elliptic regularity theory, which allows us to deduce that
	\begin{equation}
		\begin{split}
			\left\| u \right\| _{W^{2,q}}&\le C\left( \left\| \rho u_t \right\| _{L^q}+\left\| \rho v\cdot \nabla u \right\| _{L^q}+\left\| \nabla P \right\| _{L^q}+\left\| B\times \left( \nabla \times B \right) \right\| _{L^q} \right) 
			\\
			&\le C\left( c_0\left\| \nabla u_t \right\| _{L^2}+c_0c_2\left\| \nabla u \right\| _{L^q}+c_0+\left\| B \right\| _{W^{1,q}}\left\| \nabla B \right\| _{L^q} \right) 
			\\
			&\le C\left( c_0\left\| \nabla u_t \right\| _{L^2}+c_0c_2\left\| \nabla u \right\| _{L^q}+c_{0}^{2} \right) 
			\\
			&\le C\left( c_0\left\| \nabla u_t \right\| _{L^2}+c_{0}^{\frac{11}{2}}c_{1}^{2}c_2 \right) .
		\end{split}
	\end{equation}
	Then
	\begin{equation}
		\begin{split}
			\int_0^{T_3}{\left\| u \right\| _{W^{2,q}}^{2}ds}&\le C\left( \int_0^{T_3}{c_{0}^{11}c_{1}^{4}c_{2}^{2}ds}+\int_0^{T_3}{c_{0}^{2}\left\| \nabla u_t \right\| _{L^2}^{2}ds} \right) 
			\\
			&\le CT_3c_{0}^{11}c_{1}^{4}c_{2}^{2}+Cc_{0}^{7}\le Cc_{0}^{7},
		\end{split}
	\end{equation}
	provided
	\begin{equation}\label{2.42}
		T_3\overset{\mathrm{def}}{=}\min \left( T_2,\frac{1}{c_{0}^{4}c_{1}^{4}c_{2}^{2}} \right) .
	\end{equation}
	Taking $c_1=3Cc_{0}^{7},c_2=27C^4c_{0}^{\frac{37}{2}},\kappa =9C^3c_{0}^{\frac{23}{2}}$, we get 
	\begin{equation}
		\begin{split}
			\left\| u \right\| _{L^{\infty}H^1}+\kappa ^{-1}\left\| u \right\| _{L^{\infty}H^2}+\int_0^{T_*}{\left\| u_t\left( s \right) \right\| _{H^1}^{2}+\left\| u\left( s \right) \right\| _{W^{2,q}}^{2}ds}&\le Cc_{0}^{5}+\frac{Cc_{0}^{\frac{9}{2}}c_{1}^{2}}{9C^3c_{0}^{\frac{23}{2}}}+Cc_{0}^{7}
			\\
			&\le 3Cc_{0}^{7}=c_1.
		\end{split}
	\end{equation}
	Employing the method used in \cite{Cho-Choe-Kim} to prove continuity, we finally arrive at \eqref{2.8}. 
	This completes the proof of Lemma \ref{lemma 2.1}.
\end{proof}
We are now ready to prove the existence for the case of a density that allows for vacuum.
\begin{lemma}\label{lemma2.2}
	Assume that the initial data $(\rho_0,u_0,P_0,B_0)$ satisfy the regularity conditions \eqref{2.3} and the additional conditions
	\begin{equation}
		\begin{cases}
			\rho _{0}^{1/2}g=-\mu \Delta u_0-\left( \mu +\lambda \right) \nabla \left( \nabla \cdot u_0 \right) +\nabla P_0+B_0\times \left( \nabla \times B_0 \right) , \mathrm{for}\, \mathrm{some}\, g\in L^2,\\
			2+\left\| \left( \rho _0,B_0,P_0 \right) \right\| _{W^{1,q}}+\left\| u_0 \right\| _{H^2}+\left\| g \right\| _{L^2}^{2}<c_0.\\
		\end{cases}
	\end{equation}
	If the function $v$ further satisfies conditions \eqref{2.55}, \eqref{2.5} and \eqref{2.6}, then the initial-boundary value problem \eqref{2.1}, \eqref{2.2} admits a unique strong solution which satisfies \eqref{2.8} and \eqref{2.100}. Here $c_1, c_2, \kappa$ and $T_*$ are defined in \eqref{2.111} and \eqref{2.42}, respectively.
\end{lemma}
\begin{proof}
	For any $\delta \in (0,1)$, set $\rho _{0}^{\delta}=\rho _0+\delta , P_{0}^{\delta}=P_0, B_{0}^{\delta}=B_0$. Let $u_{0}^{\delta}\in H_{0}^{1}\cap H^2$ be the solution of 
	\begin{equation*}
		\begin{cases}
			\mu \Delta u_{0}^{\delta}+\left( \mu +\lambda \right) \nabla \left( \nabla \cdot u_{0}^{\delta} \right) =\nabla P_{0}^{\delta}+B_{0}^{\delta}\times \left( \nabla \times B_{0}^{\delta} \right) -\left( \rho _{0}^{\delta} \right) ^{1/2}g,\\
			u_{0}^{\delta}|_{\partial \Omega}=0.\\
		\end{cases}
	\end{equation*}
	Then it holds that 
	\begin{equation*}
		\left\| u_{0}^{\delta}-u_0 \right\| _{H^2}\le C\| \left( \left( \rho _{0}^{\delta} \right) ^{1/2}-\rho _{0}^{1/2} \right) g \| _{L^2}\rightarrow 0.
	\end{equation*}
	Moreover, for $\delta>0$ small enough, we have
	\begin{equation*}
		2+\left\| \left( \rho _{0}^{\delta},B_{0}^{\delta},P_{0}^{\delta} \right) \right\| _{W^{1,q}}+\left\| u_{0}^{\delta} \right\| _{H^2}+\left\| g \right\| _{L^2}^{2}<c_0.
	\end{equation*}
	Then it follows from Lemma \ref{lemma 2.1} that, for any small $\delta>0$, there exists unique strong solution $\left( \rho ^{\delta},u^{\delta},P^{\delta},B^{\delta} \right)$ to the linearized problem \eqref{2.1} which satisfies the estimates \eqref{2.100}. By standard compactness techniques, a subsequence of the approximate solutions $\left( \rho ^{\delta},u^{\delta},P^{\delta},B^{\delta} \right)$ converges strongly to a limit $\left( \rho ,u,P,B \right)$ in the following sense
	\begin{equation}
		\left( \rho ^{\delta},P^{{\delta}},B^{{\delta}} \right) \rightarrow \left( \rho ,P,B \right) \,\,\mathrm{in}\,\,C\left( \left[ 0,T^* \right] ;L^q \right) ,u^{\delta}\rightarrow u\,\,\mathrm{in}\,\,C\left( \left[ 0,T^* \right] ;H^1 \right) .
	\end{equation}
	It is clear that the limit $\left( \rho ,u,P,B \right)$ also obeys the a priori estimates \eqref{2.100}. Moreover, it is readily verified that $\left( \rho ,u,P,B \right)$ is a strong solution to the linearized problem \eqref{2.1} which satisfies the regularity \eqref{2.8}. 
	
	It remains to prove that the solution is unique within this regularity class. Let $(\rho_1,u_1,P_1,B_1)$ and $(\rho_2,u_2,P_2,B_2)$ be two solutions with the same initial data. Set
	$$\tilde{\rho}=\rho _1-\rho _2, \tilde{u}=u_1-u_2, \tilde{P}=P_1-P_2, \tilde{B}=B_1-B_2. $$
	Then $\tilde{\rho} \in L^{\infty}\left( 0,T_*;L^q \right)$ solves the following equation
	$$\tilde{\rho}_t+\nabla \cdot \left( \tilde{\rho}v \right) =0,$$
	so we have 
	$$\left\| \tilde{\rho}\left( t \right) \right\| _{W^{1,q}}\le \left\| \tilde{\rho}_0 \right\| _{W^{1,q}}\exp \left( C\int_0^t{\left\| \nabla v\left( s \right) \right\| _{W^{1,q}}ds} \right) =0.$$
	That is $\rho_1=\rho_2$. Similarly we can also get $B_1=B_2$ and $P_1=P_2$. The difference function $\tilde{u}$ satisfies
	\begin{equation}\label{2.35}
		\rho \tilde{u}_t+\rho v\cdot \nabla \tilde{u}-\mu \Delta \tilde{u}-\left( \mu +\lambda \right) \nabla \left( \nabla \cdot \tilde{u} \right) =0.
	\end{equation} 
	Multiplying \eqref{2.35} by $\tilde{u}$ and integrating by parts, one has 
	$$\frac{1}{2}\int{\rho \left| \tilde{u} \right|^2dx}+\int_0^{t}{\int{\mu \left| \nabla \tilde{u} \right|^2+\left( \mu +\lambda \right) \left| \nabla \cdot \tilde{u} \right|^2dxdt}=0}.$$
	We deduce that $\left\| \sqrt{\rho}\tilde{u} \right\| _{L^2}=0$ and $\left\| \nabla \tilde{u} \right\| _{L^2}=0$. This gives $u_1=u_2$.
\end{proof}
\subsection{Proof of Theorem \ref{th1}}
On the basis of the linear problem's a priori estimates, one can demonstrate the existence and regularity of a unique solution $(\rho,u,P,B)$ to the nonlinear problem.
\begin{proof}
	Let $u^{\left( 0 \right)}\in C\left( \left[ 0,\infty \right) ; H^2\cap H_{0}^{1} \right) $ be the solution to 
	\begin{equation*}
		\pi _t-\Delta \pi =0, \pi \left( 0 \right) =u_0, \pi |_{\partial \Omega}=0.
	\end{equation*}
	Then 
	\begin{equation*}
		\underset{0\le t\le T_*}{\mathrm{sup}}\left( \| u^{\left( 0 \right)}\left( t \right) \| _{H_{0}^{1}}+\kappa ^{-1}\| u^{\left( 0 \right)}\left( t \right) \| _{H^2} \right) +\int_0^{T_*}{\| u_{t}^{\left( 0 \right)}\left( s \right) \| _{H^1}^{2}+\| u^{\left( 0 \right)}\left( s \right) \| _{W^{2,q}}^{2}ds}\le c_1.
	\end{equation*}
	By lemma \ref{lemma2.2}, the linearized problem \eqref{2.1} with $v=u^{(0)}$ admits a unique strong solution $\left( \rho ^{\left( 1 \right)},u^{\left( 1 \right)},P^{\left( 1 \right)},B^{\left( 1 \right)} \right) $. By induction, we assume that $u^{(i-1)}$ has been defined for $i \ge 1$. Let $\left( \rho ^{\left( i \right)},u^{\left( i \right)},P^{\left( i \right)},B^{\left( i \right)} \right)$ be the unique solution to the linearized problem \eqref{2.1} with $v=u^{(i-1)}$.
	We conclude from Lemma \ref{lemma2.2} that there exists a constant $C$ such that
	\begin{equation}\label{2.36}
		\begin{cases}
			\left\| \rho ^{\left( i \right)} \right\| _{L^{\infty}W^{1,q}}+\| \rho _{t}^{\left( i \right)} \| _{L^{\infty}L^q}+\left\| P^{\left( i \right)} \right\| _{L^{\infty}W^{1,q}}+\| P_{t}^{\left( i \right)} \| _{L^{\infty}L^q}\\
			+\left\| B^{\left( i \right)} \right\| _{L^{\infty}W^{1,q}}+\| B_{t}^{\left( i \right)} \| _{L^{\infty}L^q}\le Cc_{2}^{2},\\
			\| \sqrt{\rho ^{\left( i \right)}}u_{t}^{\left( i \right)} \| _{L^{\infty}L^2}\le Cc_{0}^{\frac{5}{2}},\\
			\left\| u^{\left( i \right)} \right\| _{L^{\infty}H^1}+\kappa ^{-1}\left\| u^{\left( i \right)} \right\| _{L^{\infty}H^2}+\int_0^{T_*}{\| u_{t}^{\left( i \right)}\left( s \right) \| _{H^1}^{2}+\left\| u^{\left( i \right)}\left( s \right) \right\| _{W^{2,q}}^{2}ds}\le c_1.\\
		\end{cases}
	\end{equation}
	We show that the sequence $\left( \rho ^{\left( i \right)},u^{\left( i \right)},P^{\left( i \right)},B^{\left( i \right)} \right)$ converges to a solution to the nonlinear problem in a strong sense.
	Let us define 
	\begin{equation*}
		\tilde{\rho}^{\left( i+1 \right)}=\rho ^{\left( i+1 \right)}-\rho ^{\left( i \right)}, \tilde{u}^{\left( i+1 \right)}=u^{\left( i+1 \right)}-u^{\left( i \right)},\tilde{P}^{\left( i+1 \right)}=P^{\left( i+1 \right)}-P^{\left( i \right)},\tilde{B}^{\left( i+1 \right)}=B^{\left( i+1 \right)}-B^{\left( i \right)}.
	\end{equation*}
	Then we get 
	\begin{equation}\label{2.37}
		\begin{cases}
			\tilde{\rho}_{t}^{\left( i+1 \right)}+\nabla \cdot \left( \tilde{\rho}^{\left( i+1 \right)}u^{\left( i \right)} \right) +\nabla \cdot \left( \rho ^{\left( i \right)}\tilde{u}^{\left( i \right)} \right) =0,\\
			\rho ^{\left( i+1 \right)}\tilde{u}_{t}^{\left( i+1 \right)}+\rho ^{\left( i+1 \right)}u^{\left( i \right)}\cdot \nabla \tilde{u}^{\left( i+1 \right)}-\mu \Delta \tilde{u}^{\left( i+1 \right)}-\left( \mu +\lambda \right) \nabla \left( \nabla \cdot \tilde{u}^{\left( i+1 \right)} \right)\\
			=\tilde{\rho}^{\left( i+1 \right)}\left( -u_{t}^{\left( i \right)}-u^{\left( i-1 \right)}\cdot \nabla u^{\left( i \right)} \right) -\rho ^{\left( i+1 \right)}\tilde{u}^{\left( i \right)}\cdot \nabla u^{\left( i \right)}-\nabla \tilde{P}^{\left( i+1 \right)}\\
			\,\,\,\,+B^{\left( i+1 \right)}\cdot \nabla \tilde{B}^{\left( i+1 \right)}+\tilde{B}^{\left( i+1 \right)}\cdot \nabla B^{\left( i \right)}-\nabla B^{\left( i+1 \right)}\cdot B^{\left( i+1 \right)}+\nabla B^{\left( i \right)}\cdot B^{\left( i \right)},\\
			\tilde{P}_{t}^{\left( i+1 \right)}+u^{\left( i \right)}\cdot \nabla \tilde{P}^{\left( i+1 \right)}+\tilde{u}^{\left( i \right)}\cdot \nabla P^{\left( i \right)}+\gamma \tilde{P}^{\left( i+1 \right)}\nabla \cdot u^{\left( i \right)}+\gamma P^{\left( i \right)}\nabla \cdot \tilde{u}^{\left( i \right)}=0,\\
			\tilde{B}_{t}^{\left( i+1 \right)}+\nabla \times \left( \tilde{B}^{\left( i+1 \right)}\times u^{\left( i \right)}+B^{\left( i \right)}\times \tilde{u}^{\left( i \right)} \right) =0,\\
			\nabla \cdot \tilde{B}^{\left( i+1 \right)}=0.\\
		\end{cases}
	\end{equation}
	Multiplying $\eqref{2.37}_1$ by $\tilde{\rho}^{\left( i+1 \right)}$ and integrating by parts, we arrive at
	\begin{equation}\label{2.38}
		\begin{split}
			\frac{d}{dt}\int{| \tilde{\rho}^{\left( i+1 \right)}|^2dx}\le& C\int{| \nabla u^{\left( i \right)} || \tilde{\rho}^{\left( i+1 \right)}|^2dx}
			\\
			&+C\int{| \nabla \rho ^{\left( i \right)} || \tilde{u}^{\left( i \right)} || \tilde{\rho}^{\left( i+1 \right)} |+| \rho ^{\left( i \right)} || \nabla \tilde{u}^{\left( i \right)} || \tilde{\rho}^{\left( i+1 \right)} |dx}
			\\
			\le& C\| \nabla u^{\left( i \right)}\| _{W^{1,q}}\| \tilde{\rho}^{\left( i+1 \right)} \| _{L^2}^{2}
			\\
			&+C\left( \| \nabla \rho ^{\left( i \right)} \| _{L^q}+\| {\rho}^{\left( i \right)} \| _{L^{\infty}} \right) \| \nabla \tilde{u}^{\left( i \right)} \| _{L^2}\| \tilde{\rho}^{\left( i+1 \right)}\| _{L^2}
			\\
			\le& A_{\varepsilon}^{\left( i \right)}\| \tilde{\rho}^{\left( i+1 \right)} \| _{L^2}^{2}+\varepsilon \| \nabla \tilde{u}^{\left( i \right)} \| _{L^2}^{2},
		\end{split}
	\end{equation}
	where 
	$$A_{\varepsilon}^{\left( i \right)}\overset{\mathrm{def}}{=}C\| \nabla u^{\left( i \right)} \| _{W^{1,q}}+\frac{C}{\varepsilon}\left( \| \nabla \rho ^{\left( i \right)} \| _{L^q}^{2}+\| \rho^{\left( i \right)} \| _{L^{\infty}}^{2} \right).$$
	Multiplying $\eqref{2.37}_3$ by $\tilde{P}^{\left( i+1 \right)}$ and integrating by parts yields 
	\begin{equation}\label{2.39}
		\begin{split}
			\frac{d}{dt}\int{|\tilde{P}^{\left( i+1 \right)}|^2dx}\le& C\parallel \nabla u^{\left( i \right)}\parallel _{W^{1,q}}\parallel \tilde{P}^{\left( i+1 \right)}\parallel _{L^2}^{2}
			\\
			&+C\left( \parallel \nabla P^{\left( i \right)}\parallel _{L^q}+\parallel P^{\left( i \right)}\parallel _{L^{\infty}} \right) \parallel \nabla \tilde{u}^{\left( i \right)}\parallel _{L^2}\parallel \tilde{P}^{\left( i+1 \right)}\parallel _{L^2}
			\\
			\le& B_{\varepsilon}^{\left( i \right)}\parallel \tilde{P}^{\left( i+1 \right)}\parallel _{L^2}^{2}+\varepsilon \parallel \nabla \tilde{u}^{\left( i \right)}\parallel _{L^2}^{2},
		\end{split}	
	\end{equation}
	where $B_{\varepsilon}^{\left( i \right)}$ is defined as 
	\begin{equation*}
		B_{\varepsilon}^{\left( i \right)}\overset{\mathrm{def}}{=}C\| \nabla u^{\left( i \right)} \| _{W^{1,q}}+\frac{C}{\varepsilon}\left( \| \nabla P^{\left( i \right)} \| _{L^q}^{2}+\| P^{\left( i \right)} \| _{L^{\infty}}^{2} \right).
	\end{equation*}
	Taking $L^2$ inner product of $\eqref{2.37}_4$ with $\tilde{B}^{\left( i+1 \right)}$, similarly we get 
	\begin{equation}\label{2.40}
		\frac{d}{dt}\int{|\tilde{B}^{\left( i+1 \right)}|^2dx}\le C_{\varepsilon}^{\left( i \right)}\parallel \tilde{B}^{\left( i+1 \right)}\parallel _{L^2}^{2}+\varepsilon \parallel \nabla \tilde{u}^{\left( i \right)}\parallel _{L^2}^{2},
	\end{equation} 
	where 
	\begin{equation*}
		C_{\varepsilon}^{\left( i \right)}\overset{\mathrm{def}}{=}C\| \nabla u^{\left( i \right)} \| _{W^{1,q}}+\frac{C}{\varepsilon}\left( \| \nabla B^{\left( i \right)} \| _{L^q}^{2}+\| B^{\left( i \right)} \| _{L^{\infty}}^{2} \right).
	\end{equation*}
	Multiplying $\eqref{2.37}_4$ by $\tilde{u}^{\left( i+1 \right)}$ and integrating by parts, we deduce that
	\begin{equation*}
		\begin{split}
			\frac{1}{2}\frac{d}{dt}&\int{\rho ^{\left( i+1 \right)}| \tilde{u}^{\left( i+1 \right)}|^2dx}+\mu \int{| \nabla \tilde{u}^{\left( i+1 \right)}|^2dx}+\left( \mu +\lambda \right) \int{| \nabla \cdot \tilde{u}^{\left( i+1 \right)}|^2dx}
			\\
			\le& C\int{| \tilde{\rho}^{\left( i+1 \right)} |\left( | u_{t}^{\left( i \right)} |+| u^{\left( i-1 \right)}||\nabla u^{\left( i \right)} | \right) | \tilde{u}^{\left( i+1 \right)} |dx}
			\\
			&+C\int{| \rho ^{\left( i+1 \right)} || \tilde{u}^{\left( i \right)} || \nabla u^{\left( i \right)}|| \tilde{u}^{\left( i+1 \right)}|+| \tilde{P}^{\left( i+1 \right)} || \nabla \tilde{u}^{\left( i+1 \right)}|dx}
			\\
			&+C\int{\left( | B^{\left( i \right)} |+| B^{\left( i+1 \right)} | \right) | \tilde{B}^{\left( i+1 \right)} || \nabla \tilde{u}^{\left( i+1 \right)}|dx}
			\\
			\le& C\| \tilde{\rho}^{\left( i+1 \right)} \| _{L^2}\left( \| u_{t}^{\left( i \right)} \| _{L^4}+\| u^{\left( i-1 \right)} \| _{L^{\infty}}\| \nabla u^{\left( i \right)}\| _{L^4} \right)\| \tilde{u}^{\left( i+1 \right)} \| _{L^4}
			\\
			&+C\| \rho ^{\left( i+1 \right)} \| _{L^{\infty}}^{1/2}\| \tilde{u}^{\left( i \right)} \| _{L^4}\| \nabla u^{\left( i \right)} \| _{L^4}\| \sqrt{\rho ^{\left( i+1 \right)}}\tilde{u}^{\left( i+1 \right)} \| _{L^2}
			\\
			&+C\| \tilde{P}^{\left( i+1 \right)} \| _{L^2}\| \nabla \tilde{u}^{\left( i+1 \right)} \| _{L^2}+C\left( \| B^{\left( i \right)} \| _{L^{\infty}}+\| B^{\left( i+1 \right)} \| _{L^{\infty}} \right) \| \tilde{B}^{\left( i+1 \right)} \| _{L^2}\| \nabla \tilde{u}^{\left( i+1 \right)} \| _{L^2}
			\\
			\le& C\| \tilde{\rho}^{\left( i+1 \right)} \| _{L^2}\| \nabla \tilde{u}^{\left( i+1 \right)} \| _{L^2}\left( \| u_{t}^{\left( i \right)} \| _{L^4}+1 \right) +C{\| \nabla \tilde{u}^{\left( i \right)} \| _{L^2}}\| \sqrt{\rho ^{\left( i+1 \right)}}\tilde{u}^{\left( i+1 \right)} \| _{L^2}
			\\
			&+C\| \tilde{P}^{\left( i+1 \right)} \| _{L^2}\| \nabla \tilde{u}^{\left( i+1 \right)} \| _{L^2}+C\| \tilde{B}^{\left( i+1 \right)} \| _{L^2}\| \nabla \tilde{u}^{\left( i+1 \right)} \| _{L^2}.
		\end{split}
	\end{equation*}
	Using Young's inequality, one can get
	\begin{equation*}
		\begin{split}
			\frac{1}{2}\frac{d}{dt}\| \sqrt{\rho ^{\left( i+1 \right)}}\tilde{u}^{\left( i+1 \right)} \| _{L^2}^{2}+\mu \| \nabla \tilde{u}^{\left( i+1 \right)} \| _{L^2}^{2}\le& C\| \tilde{\rho}^{\left( i+1 \right)} \| _{L^2}^{2}\left( \| u_{t}^{\left( i \right)} \| _{L^4}+1 \right) ^2
			\\
			&+\frac{C}{\varepsilon}\| \sqrt{\rho ^{\left( i+1 \right)}}\tilde{u}^{\left( i+1 \right)} \|^2 _{L^2}+C\| \tilde{P}^{\left( i+1 \right)} \| _{L^2}^{2}
			\\
			&+C\| \tilde{B}^{\left( i+1 \right)} \| _{L^2}^{2}+\frac{\mu}{2}\| \nabla \tilde{u}^{\left( i+1 \right)} \| _{L^2}^{2}+\varepsilon \| \nabla \tilde{u}^{\left( i \right)} \| _{L^2}^{2}.
		\end{split}
	\end{equation*}
	Hence, we arrive at
	\begin{equation}\label{2.41}
		\begin{split}
			\frac{1}{2}\frac{d}{dt}\parallel \sqrt{\rho ^{\left( i+1 \right)}}\tilde{u}^{\left( i+1 \right)}\parallel _{L^2}^{2}+\frac{\mu}{2}\parallel \nabla \tilde{u}^{\left( i+1 \right)}\parallel _{L^2}^{2}\le& C\parallel \tilde{\rho}^{\left( i+1 \right)}\parallel _{L^2}^{2}\left( \parallel u_{t}^{\left( i \right)}\parallel _{L^4}+1 \right) ^2
			\\
			&+\frac{C}{\varepsilon}\parallel \sqrt{\rho ^{\left( i+1 \right)}}\tilde{u}^{\left( i+1 \right)}\parallel ^2_{L^2}+C\parallel \tilde{P}^{\left( i+1 \right)}\parallel _{L^2}^{2}
			\\
			&+C\parallel \tilde{B}^{\left( i+1 \right)}\parallel _{L^2}^{2}+\varepsilon \parallel \nabla \tilde{u}^{\left( i \right)}\parallel _{L^2}^{2}.
		\end{split}
	\end{equation}
	In summary, \eqref{2.38}-\eqref{2.41} show that
	\begin{equation}
		\begin{cases}
			\frac{d}{dt}\parallel \tilde{\rho}^{\left( i+1 \right)}\parallel _{L^2}^{2}\le A_{\varepsilon}^{\left( i \right)}\parallel \tilde{\rho}^{\left( i+1 \right)}\parallel _{L^2}^{2}+\varepsilon \parallel \nabla \tilde{u}^{\left( i \right)}\parallel _{L^2}^{2},\\
			\frac{d}{dt}\parallel \tilde{P}^{\left( i+1 \right)}\parallel _{L^2}^{2}\le B_{\varepsilon}^{\left( i \right)}\parallel \tilde{P}^{\left( i+1 \right)}\parallel _{L^2}^{2}+\varepsilon \parallel \nabla \tilde{u}^{\left( i \right)}\parallel _{L^2}^{2},\\
			\frac{d}{dt}\parallel \tilde{B}^{\left( i+1 \right)}\parallel _{L^2}^{2}\le C_{\varepsilon}^{\left( i \right)}\parallel \tilde{B}^{\left( i+1 \right)}\parallel _{L^2}^{2}+\varepsilon \parallel \nabla \tilde{u}^{\left( i \right)}\parallel _{L^2}^{2},\\
			\frac{d}{dt}\parallel \sqrt{\rho ^{\left( i+1 \right)}}\tilde{u}^{\left( i+1 \right)}\parallel _{L^2}^{2}+\mu \parallel \nabla \tilde{u}^{\left( i+1 \right)}\parallel _{L^2}^{2}\\
			\,\, \,\,    \le C\parallel \tilde{\rho}^{\left( i+1 \right)}\parallel _{L^2}^{2}\left( \parallel u_{t}^{\left( i \right)}\parallel _{L^4}+1 \right) ^2+\frac{C}{\varepsilon}\parallel \sqrt{\rho ^{\left( i+1 \right)}}\tilde{u}^{\left( i+1 \right)}\parallel _{L^2}\\
			\,\,\,\,\,\,\,\,        +C\parallel \tilde{P}^{\left( i+1 \right)}\parallel _{L^2}^{2}+C\parallel \tilde{B}^{\left( i+1 \right)}\parallel _{L^2}^{2}+\varepsilon \parallel \nabla \tilde{u}^{\left( i \right)}\parallel _{L^2}^{2},\\
		\end{cases}
	\end{equation}
	with
	\begin{equation}
		\begin{cases}
			A_{\varepsilon}^{\left( i \right)}\overset{\mathrm{def}}{=}C\parallel \nabla u^{\left( i \right)}\parallel _{W^{1,q}}+\frac{C}{\varepsilon}\left( \parallel \nabla \rho ^{\left( i \right)}\parallel _{L^q}^{2}+\parallel \rho ^{\left( i \right)}\parallel _{L^{\infty}}^{2} \right) ,\\
			B_{\varepsilon}^{\left( i \right)}\overset{\mathrm{def}}{=}C\parallel \nabla u^{\left( i \right)}\parallel _{W^{1,q}}+\frac{C}{\varepsilon}\left( \parallel \nabla P^{\left( i \right)}\parallel _{L^q}^{2}+\parallel P^{\left( i \right)}\parallel _{L^{\infty}}^{2} \right) ,\\
			C_{\varepsilon}^{\left( i \right)}\overset{\mathrm{def}}{=}C\parallel \nabla u^{\left( i \right)}\parallel _{W^{1,q}}+\frac{C}{\varepsilon}\left( \parallel \nabla B^{\left( i \right)}\parallel _{L^q}^{2}+\parallel B^{\left( i \right)}\parallel _{L^{\infty}}^{2} \right) .\\
		\end{cases}
	\end{equation}
	Define 
	$$\varphi ^{\left( i+1 \right)}\left( t \right) \overset{\mathrm{def}}{=}\parallel \tilde{\rho}^{\left( i+1 \right)}\parallel _{L^2}^{2}+\parallel \tilde{P}^{\left( i+1 \right)}\parallel _{L^2}^{2}+\parallel \tilde{B}^{\left( i+1 \right)}\parallel _{L^2}^{2}+\parallel \sqrt{\rho ^{\left( i+1 \right)}}\tilde{u}^{\left( i+1 \right)}\parallel _{L^2}^{2}.
	$$
	Then we get 
	\begin{equation*}
		\begin{split}
			\frac{d}{dt}&\varphi ^{\left( i+1 \right)}\left( t \right) +\mu \parallel \nabla \tilde{u}^{\left( i+1 \right)}\parallel _{L^2}^{2}
			\\
			&\le C\left( A_{\varepsilon}^{\left( i \right)}+B_{\varepsilon}^{\left( i \right)}+C_{\varepsilon}^{\left( i \right)}+\frac{C}{\varepsilon}+\parallel u_{t}^{\left( i \right)}\parallel _{L^4}^{2}+1 \right) \varphi ^{\left( i+1 \right)}\left( t \right) +\varepsilon \parallel \nabla \tilde{u}^{\left( i \right)}\parallel _{L^2}^{2}.
		\end{split}
	\end{equation*}
	Given that $\varphi ^{\left( i+1 \right)}\left( 0 \right)=0$, it follows from Gronwall's inequality that
	\begin{equation}
		\underset{0\le t\le T_*}{\mathrm{sup}}\varphi ^{\left( i+1 \right)}\left( t \right) +\mu \int_0^{T_*}{\parallel \nabla \tilde{u}^{\left( i+1 \right)}\parallel _{L^2}^{2}dt}\le \varepsilon \int_0^{T_*}{\parallel \nabla \tilde{u}^{\left( i \right)}\parallel _{L^2}^{2}dt}\cdot \exp \left( \int_0^{T_*}{D_{\varepsilon}^{\left( i \right)}\left( t \right) dt} \right) ,
	\end{equation}
	where 
	$$D_{\varepsilon}^{\left( i \right)}\left( t \right) \overset{\mathrm{def}}{=}C\left( A_{\varepsilon}^{\left( i \right)}+B_{\varepsilon}^{\left( i \right)}+C_{\varepsilon}^{\left( i \right)}+\frac{C}{\varepsilon}+\parallel u_{t}^{\left( i \right)}\parallel _{L^4}^{2}+1 \right) .$$
	One deduces from \eqref{2.36} that 
	\begin{equation*}
		\begin{split}
			\int_0^{T_*}{D_{\varepsilon}^{\left( i \right)}\left( t \right) dt}&\le C\left( \int_0^{T_*}{\left\| u \right\| _{W^{2,q}}^{2}dt} \right) ^{\frac{1}{2}}\left( \int_0^{T_*}{1dt} \right) ^{\frac{1}{2}}+\frac{C}{\varepsilon}T_*+C\int_0^{T_*}{\parallel u_{t}^{\left( i \right)}\parallel _{H^1}^{2}dt}+CT_*
			\\
			&\le C\sqrt{T_*}+\frac{C}{\varepsilon}T_*+C+CT_*.
		\end{split}
	\end{equation*}
	Define
	\begin{equation}
		T^*\overset{\mathrm{def}}{=}\min \left\{ T_*,\varepsilon \right\} ,
	\end{equation}
	so we have
	\begin{equation}\label{2.46}
		\underset{0\le t\le T^*}{\mathrm{sup}}\varphi ^{\left( i+1 \right)}\left( t \right) +\mu \int_0^{T^*}{\parallel \nabla \tilde{u}^{\left( i+1 \right)}\parallel _{L^2}^{2}dt}\le \varepsilon \exp \left\{ C \right\} \int_0^{T^*}{\parallel \nabla \tilde{u}^{\left( i \right)}\parallel _{L^2}^{2}dt}.
	\end{equation}
	Now, taking $\varepsilon$ suitably small such that
	$$\varepsilon \exp \left( C \right) \le \frac{1}{2}\mu. $$
	By summing \eqref{2.46} over $i$ from 1 to $\infty$, we obtain
	\begin{equation}
		\sum_{i=1}^{\infty}{\underset{0\le t\le T^*}{\mathrm{sup}}\varphi ^{\left( i+1 \right)}\left( t \right)}+\sum_{i=1}^{\infty}{\mu \int_0^{T^*}{\parallel \nabla \tilde{u}^{\left( i+1 \right)}\parallel _{L^2}^{2}dt}}\le \sum_{i=1}^{\infty}{\frac{1}{2}\mu \int_0^{T^*}{\parallel \nabla \tilde{u}^{\left( i \right)}\parallel _{L^2}^{2}dt}}.
	\end{equation}
	Hence 
	\begin{equation}
		\sum_{i=1}^{\infty}{\underset{0\le t\le T^*}{\mathrm{sup}}\varphi ^{\left( i+1 \right)}\left( t \right)}+\sum_{i=1}^{\infty}{\frac{1}{2}\mu \int_0^{T^*}{\parallel \nabla \tilde{u}^{\left( i+1 \right)}\parallel _{L^2}^{2}dt}}\le \frac{1}{2}\mu \int_0^{T^*}{\parallel \nabla \tilde{u}^{\left( 1 \right)}\parallel _{L^2}^{2}dt}\le C.
	\end{equation}
	Therefore, we conclude that $\left( \rho ^{\left( i \right)},u^{\left( i \right)},P^{\left( i \right)},B^{\left( i \right)} \right)$ converges to a limit $(\rho, u, P, B)$ in the following strong sense
	\begin{equation*}
		\left( \rho ^{\left( i \right)},P^{\left( i \right)},B^{\left( i \right)} \right) \rightarrow \left( \rho ,P,B \right) \,\,\mathrm{in}\,\, L^{\infty}\left( 0,T^*;L^2 \right) ,u^{\left( i \right)}\rightarrow u\,\,\mathrm{in}\,\, L^2\left( 0,T^*;H^1 \right) .
	\end{equation*}
	A standard argument shows that $(\rho,u,P,B)$ is a weak solution to the nonlinear problem. Due to the uniform estimates \eqref{2.36}, the limit functions $(\rho, u, P, B)$ inherit the regularity properties of the approximate sequence 
	\begin{equation}
		\begin{split}
			\underset{0\le t\le T^*}{\mathrm{sup}}&\left\| \sqrt{\rho}u_t \right\| _{L^2}+\int_0^{T^*}{\left\| u_t \right\| _{H^1}^{2}+\left\| u \right\| _{W^{2,q}}^{2}dt}
			\\
			&+\underset{0\le t\le T^*}{\mathrm{sup}}\left( \left\| \left( \rho , P, B \right) \right\| _{W^{1,q}}+\left\| \left( \rho _t, P_t, B_t \right) \right\| _{L^q}+\left\| u \right\| _{H^2} \right) \le C.
		\end{split}
	\end{equation}
	Consequently, the limit pair $(\rho, u, P, B)$ satisfies the equations almost everywhere. Thus, we conclude that $(\rho, u, P, B)$ is the strong solution to the nonlinear problem on the time interval $[0, T^*]$. Finally, we prove the uniqueness of the strong solutions.
	\par
		Let $(\rho_1, u_1, P_1, B_1)$ and $(\rho_2, u_2, P_2, B_2)$ be two strong solutions to the problem \eqref{1.2} in the class defined by \eqref{2.8} and \eqref{2.100} with the same initial data $(\rho_0, u_0, P_0, B_0)$.
		
		Define the difference functions
		\begin{equation*}
			\tilde{\rho} = \rho_1 - \rho_2, \quad \tilde{u} = u_1 - u_2, \quad \tilde{P} = P_1 - P_2, \quad \tilde{B} = B_1 - B_2.
		\end{equation*}
	Then we get the error system
		\begin{equation}\label{diff_sys}
			\begin{cases}
				\tilde{\rho}_t + \nabla \cdot (\tilde{\rho} u_1) + \nabla \cdot (\rho_2 \tilde{u}) = 0, \\
				\rho_1 \tilde{u}_t + \rho_1 u_1 \cdot \nabla \tilde{u} - \mu \Delta \tilde{u} - (\mu + \lambda) \nabla (\nabla \cdot \tilde{u}) \\
				\quad = -\tilde{\rho} (u_2)_t - \tilde{\rho} u_2 \cdot \nabla u_2 - \rho_1 \tilde{u} \cdot \nabla u_2 - \nabla \tilde{P} \\
				\quad \quad +(B_1 \cdot \nabla)\tilde{B} + (\tilde{B} \cdot \nabla)B_2 - \frac{1}{2}\nabla (\tilde{B} \cdot (B_1 + B_2)), \\
				\tilde{P}_t + u_1 \cdot \nabla \tilde{P} + \gamma P_1 \nabla \cdot \tilde{u} = -\tilde{u} \cdot \nabla P_2 - \gamma \tilde{P} \nabla \cdot u_2, \\
				\tilde{B}_t - \nabla \times (u_1 \times \tilde{B}) - \nabla \times (\tilde{u} \times B_2) = 0,
				\\
				\nabla \cdot \tilde{B}=0.
			\end{cases}
		\end{equation}
		Multiplying the first equation of \eqref{diff_sys} by $\tilde{\rho}$ and integrating by parts, we have
		\begin{equation}\label{5.51}
			\begin{split}
					\frac{1}{2}\frac{d}{dt}\| \tilde{\rho}\| _{L^2}^{2}&=-\int{\nabla}\cdot (\tilde{\rho}u_1)\tilde{\rho}\,dx-\int{\nabla}\cdot (\rho _2\tilde{u})\tilde{\rho}\,dx\\
					&=\frac{1}{2}\int{(}\nabla \cdot u_1)|\tilde{\rho}|^2\,dx-\int{(}\nabla \rho _2\cdot \tilde{u}+\rho _2\nabla \cdot \tilde{u})\tilde{\rho}\,dx\\
					&\le C\| \nabla u_1\| _{L^{\infty}}\| \tilde{\rho}\| _{L^2}^{2}+C\left( \| \nabla \rho _2\| _{L^q}\| \tilde{\rho}\| _{L^2}\| \nabla \tilde{u}\| _{L^2}+\| \rho _2\| _{L^{\infty}}\| \tilde{\rho}\| _{L^2}\| \nabla \tilde{u}\| _{L^2} \right)\\
					&\le C\left( \| \nabla u_1\| _{L^{\infty}}+\frac{\| \rho _2\| _{W^{1,q}}^{2}}{\varepsilon} \right) \| \tilde{\rho}\| _{L^2}^{2}+\varepsilon \| \nabla \tilde{u}\| _{L^2}^{2},
			\end{split}
		\end{equation}
		where $\varepsilon>0$ is a small constant to be determined.
		Similarly, for the pressure $\tilde{P}$, we arrive at
		\begin{equation}
			\begin{split}
				\frac{1}{2}\frac{d}{dt}\| \tilde{P}\| _{L^2}^{2}&\le C\left( \| \nabla u_1\| _{L^{\infty}}+\frac{\| P_2\| _{W^{1,q}}^{2}}{\varepsilon} \right) \| \tilde{P}\| _{L^2}^{2}+\varepsilon \| \nabla \tilde{u}\| _{L^2}^{2}.
			\end{split}
		\end{equation}
		Taking $L^2$ inner product of the fourth equation of $\eqref{diff_sys}$ with $\tilde{B}$, we obtain
	\begin{equation}
		\begin{split}
			\frac{1}{2}\frac{d}{dt}\| \tilde{B}\| _{L^2}^{2}&=-\int{(u_1\cdot \nabla \tilde{B})\cdot \tilde{B}\,dx}-\int{(\tilde{u}\cdot \nabla B_2)\cdot \tilde{B}\,dx}
			\\
			&\quad +\int{((B_1\cdot \nabla)\tilde{u}+(\tilde{B}\cdot \nabla)u_2)\cdot \tilde{B}\,dx}
			\\
			&\quad -\int{(B_1(\nabla \cdot \tilde{u})+\tilde{B}(\nabla \cdot u_2))\cdot \tilde{B}\,dx}
			\\
			&=\frac{1}{2}\int{(\nabla \cdot u_1)|\tilde{B}|^2\,dx}+\int{(\tilde{B}\cdot \nabla u_2)\cdot \tilde{B}\,dx}-\int{|\tilde{B}|^2(\nabla \cdot u_2)\,dx}
			\\
			&\quad +\int{((B_1\cdot \nabla)\tilde{u}-B_1(\nabla \cdot \tilde{u})-(\tilde{u}\cdot \nabla)B_2)\cdot \tilde{B}\,dx}
			\\
			&\le C\left( \| \nabla u_1\| _{L^{\infty}}+\| \nabla u_2\| _{L^{\infty}} \right) \| \tilde{B}\| _{L^2}^{2}
			\\
			&\quad +C\left( \| B_1\| _{L^{\infty}}+\| \nabla B_2\| _{L^q} \right) \| \nabla \tilde{u}\| _{L^2}\| \tilde{B}\| _{L^2}
			\\
			&\le C\left( \| \nabla u_1\| _{L^{\infty}}+\| \nabla u_2\| _{L^{\infty}}+\frac{\| B_1\| _{L^{\infty}}^{2}+\| \nabla B_2\| _{L^q}^{2}}{\varepsilon} \right) \| \tilde{B}\| _{L^2}^{2}+\varepsilon \| \nabla \tilde{u}\| _{L^2}^{2}.
		\end{split}
	\end{equation}
		Now multiplying the second equation of \eqref{diff_sys} by $\tilde{u}$ and integrating by parts
\begin{equation}\label{5.54}
	\begin{split}
		\frac{1}{2}\frac{d}{dt}&\int{\rho _1| \tilde{u}|^2dx}+\mu \| \nabla \tilde{u}\| _{L^2}^{2}+\left( \mu +\lambda \right) \| \nabla \cdot \tilde{u}\| _{L^2}^{2}
		\\
		=&-\int{\tilde{\rho}\left( (u_2)_t+u_2\cdot \nabla u_2 \right) \cdot \tilde{u}\,dx}-\int{\rho _1(\tilde{u}\cdot \nabla u_2)\cdot \tilde{u}\,dx}
		\\
		&+\int{\tilde{P}\nabla \cdot \tilde{u}\,dx}
		\\
		&+\int{\left( -(B_1\cdot \nabla)\tilde{u}\cdot \tilde{B}+(\tilde{B}\cdot \nabla)B_2\cdot \tilde{u} \right) \,dx}+\int{\left( \tilde{B}\cdot \frac{B_1+B_2}{2} \right) \nabla \cdot \tilde{u}\,dx}
		\\
		\le& \| \tilde{\rho}\| _{L^2}\| (u_2)_t+u_2\cdot \nabla u_2\| _{L^q}\| \nabla \tilde{u}\| _{L^2}+\| \nabla u_2\| _{L^{\infty}}\| \sqrt{\rho_1}\tilde{u}\| _{L^2}^{2}
		\\
		&+\| \tilde{P}\| _{L^2}\| \nabla \tilde{u}\| _{L^2}
		\\
		&+\| B_1\| _{L^{\infty}}\| \nabla \tilde{u}\| _{L^2}\| \tilde{B}\| _{L^2}+\| \tilde{B}\| _{L^2}\| \nabla B_2\| _{L^q}\| \nabla \tilde{u}\| _{L^2}
		\\
		&+\frac{1}{2}\left( \| B_1\| _{L^{\infty}}+\| B_2\| _{L^{\infty}} \right) \| \tilde{B}\| _{L^2}\| \nabla \tilde{u}\| _{L^2}
		\\
		\le& C\| \nabla u_2\| _{L^{\infty}}\| \sqrt{\rho_1}\tilde{u}\| _{L^2}^{2}+C_{\varepsilon}\left( \| \tilde{\rho}\| _{L^2}^{2}+\| \tilde{P}\| _{L^2}^{2}+\| \tilde{B}\| _{L^2}^{2} \right) +\varepsilon \| \nabla \tilde{u}\| _{L^2}^{2},
	\end{split}
\end{equation}
where $C_{\varepsilon}$ depends on $\left\| \left( u_2 \right) _t \right\| _{L^q},\left\| u_2 \right\| _{W^{1,q}},\left\| B_1 \right\| _{W^{1,q}},\left\| B_2 \right\| _{W^{1,q}}$.
\par
		Combining \eqref{5.51}-\eqref{5.54}, and let $\varepsilon=\mu/8$, we get
		\begin{equation}
			\begin{split}
			\frac{d}{dt}&\int{\rho _1|\tilde{u}|^2dx}+\frac{d}{dt}\left( \| \tilde{\rho}\| _{L^2}^{2}+\| \tilde{P}\| _{L^2}^{2}+\| \tilde{B}\| _{L^2}^{2} \right) +\mu \| \nabla \tilde{u}\| _{L^2}^{2}
			\\
			\le& C\left( \| \tilde{\rho}\| _{L^2}^{2}+\| \tilde{P}\| _{L^2}^{2}+\| \tilde{B}\| _{L^2}^{2} \right) 
			\\
			&+C\left( \| \nabla u_1\| _{L^{\infty}}+\| \nabla u_2\| _{L^{\infty}} \right) \left( \| \tilde{\rho}\| _{L^2}^{2}+\| \tilde{P}\| _{L^2}^{2}+\| \tilde{B}\| _{L^2}^{2} +\|\sqrt{\rho_1}\tilde{u}\| _{L^2}^{2}\right) .
			\end{split}
		\end{equation}
		Using Gronwall inequality and the integrability of $\| \nabla u_1\| _{L^{\infty}}+\| \nabla u_2\| _{L^{\infty}}$, we finally get $\rho_1 = \rho_2$, $P_1 = P_2$, $B_1 = B_2$, and $u_1 = u_2$ almost everywhere. The proof is complete.
\end{proof}	
	\section*{Acknowledgments}
X.-D. Huang is partially supported by CAS Project for Young Scientists in Basic Research
(Grant No. YSBR-031), NNSFC (Grant Nos. 12494542, 11688101) and National Key R\&D Program of China (Grant No. 2021YFA1000800).

	\bibliographystyle{siam}

\begin{thebibliography}{10} 
		\bibitem{Cho-Choe-Kim}
		{Y. Cho, H.J. Choe, H. Kim}, {Unique solvability of the initial boundary value problems for compressible viscous fluids}, J. Math. Pures Appl. 83(9) (2004) 243–275.
			\bibitem{DiPerna-Lions}
		{R.J. DiPerna, P.L. Lions}, {Ordinary differential equations, transport theory and Sobolev spaces}, Invent. Math. 98 (1989) 511–547.
		\bibitem{Fan-Gu-Huang}
		{J. Fan, Y.T. Gu, X.D. Huang}, {Global solvability of large initial data to three-dimensional compressible MHD equations with density-dependent viscosities}, Nonlinear Anal. Real World Appl. 87 (2026) 104435. 
		\bibitem{Feireisl-Novotný-Petzeltová}
		{E. Feireisl, A. Novotný, H. Petzeltová}, {On the existence of globally deﬁned weak solutions to the 
		Navier–Stokes equations of compressible isentropic fluids}, J. Math. Fluid Mech. 3 (2001) 358–392.
		\bibitem{Fan-Yu}
		{J. Fan, W. Yu}, {Strong solution to the compressible magnetohydrodynamic equations with vacuum}, Nonlinear Anal. Real World Appl. 10 (2009) 392–409.
		\bibitem{Huang-Li-Xin}
		{X.D. Huang, J. Li, Z.P. Xin}, {Global well-posedness of classical solutions with large oscillations and vacuum to the three-dimensional isentropic compressible Navier-Stokes equations}, Commun. Pure. Appl. Math. 65(4) (2012) 549–585.
		\bibitem{Huang-Yan-Xin}
		{X.D. Huang, Z.P. Xin, W. Yan}, {Finite time blowup of strong solutions to the two dimensional MHD equations},  
		Math. Ann. 392 (2025) 2365–2394.
		\bibitem{klein}
		{R. Klein}, {An applied mathematical view of meteorological modelling, in: Applied Mathematics Entering the 21st 
		Century}, SIAM, Philadelphia, PA (2004) pp. 227––269.
		\bibitem{Li-Wang-Xin}
		 {H.L. Li, Y. Wang, Z.P. Xin}, {Non-existence of Classical Solutions with Finite Energy to the Cauchy Problem of the Compressible Navier–Stokes Equations}, Arch. Ration. Mech. Anal. 232 (2019) 557–590.
		 \bibitem{Li-Xu-Zhang}
		 {H.L. Li, X.Y. Xu, J.W. Zhang}, {Global classical solutions to 3D compressible magnetohydrodynamic equations with large oscillations and vacuum}, SIAM J. Math. Anal. 45 (2013) 1356–1387.
		  \bibitem{Li-Zhang}
		  {H.L. Li, X. Zhang}, {Global strong solutions to radial symmetric compressible Navier-Stokes equations with free boundary}, J. Differ. Equations, 261(11) (2016) 6341-6367.
		 \bibitem{Merle-Raphael-p-J}
		 {F. Merle, P.Raphaël, I. Rodnianski, J.Szeftel}, {On the implosion of a compressible ﬂuid II: Singularity formation}, Ann. of Math. 196(2) (2022) 779–889.
		 \bibitem{Maltese-Michálek-Novotný}
		  {D. Maltese, M. Michálek, P.B. Mucha, A. Novotný, M. Pokorný, E. Zatorska}, {Existence of weak solutions for compressible Navier-Stokes equations with entropy transport}, J. Differ. Equ. 261(8) (2016) 4448–4485. 
		  \bibitem{Michálek}
		 {M. Michálek}, {Stability result for Navier–Stokes equations with entropy transport}, J. Math. Fluid Mech. 17(2) 
		 (2015) 279–285.
		 \bibitem{M. Lukáová-Medvid’ová-Andreas Schömer}
		 {M. Luk\'a\v{c}ov\'a-Medvid'ov\'a, A. Schömer}, {Conditional regularity for the compressible 
		 Navier-Stokes equations with potential temperature transport}, J. Differ. Equ. 423 (2025) 1-40.
		 \bibitem{Vol'pert-Hudjaev}
		 {A.I. Vol’pert, S.I. Hudjaev}, {The Cauchy problem for composite systems of nonlinear differential equations}, Math. USSR-Sb. 16 (1972) 517–544.
		  \bibitem{Xin}
		 {Z.P. Xin}, {Blowup of smooth solutions to the compressible Navier–Stokes equation with compact density}, Commun. Pure Appl. Math. 51 (1998) 229–240.
		\bibitem{Xin-Yan}
		 {Z.P. Xin, W. Yan}, {On blowup of classical solutions to the compressible Navier–Stokes equations}, Commun. Math. Phys. 321(2) (2013) 529–541.
	
	\end{thebibliography}

\end{document}